\documentclass[a4paper, reqno, 12pt]{amsart}

\usepackage[english]{babel}
\usepackage{amsmath}
\usepackage{amssymb}
\usepackage{enumerate}
\usepackage{ifthen}
\usepackage{bbm}
\usepackage{color}
\usepackage{latexsym}
\provideboolean{shownotes} 
\setboolean{shownotes}{true}
\usepackage{hyperref}
\usepackage{graphicx}
\usepackage[overload]{empheq}

\usepackage{cancel}

\usepackage{geometry}
\usepackage{inputenc}

\usepackage{amsaddr}

\newcommand{\margnote}[1]{
\ifthenelse{\boolean{shownotes}}%
{\marginpar{\raggedright\tiny\texttt{#1}}}%
{}%
}

\newtheorem{theorem}{Theorem}[section]

\newtheorem{lemma}[theorem]{Lemma}

\newtheorem{remark}[theorem]{Remark}

\newenvironment{nouppercase}{%
\renewcommand{\uppercasenonmath}[1]{}}{}

\makeatletter
\def\@settitle{\begin{center}%
  \baselineskip14\p@\relax
    \normalfont\LARGE
\uppercasenonmath\@title
  \@title
  \end{center}%
}
\makeatother

\usepackage[style=numeric-comp,maxalphanames=5,backend=bibtex]{biblatex}
\DeclareNameAlias{author}{last-first} 
\begin{document}

\author{Diego C\'ordoba$^*$ and Elena Di Iorio$^\dag$}
\address{Instituto de Ciencias Matem\'aticas\\Consejo Superior de Investigaciones Cient\'ificas \\
28049, Madrid, Spain}
\email{$^*$dcg@icmat.es, \\ $^\dag$elena.di.iorio@icmat.es}

\title{Existence of gravity-capillary Crapper waves with concentrated vorticity}

\begin{nouppercase}
\maketitle
\end{nouppercase}

\begin{abstract}
The aim of this paper is to prove the existence of gravity-capillary Crapper waves with the presence of vorticity. In particular, we consider a concentrated vorticity: point vortex and vortex patch. We show that for small gravity and small vorticity it is possible to demonstrate that the waves are overhanging.
\end{abstract}

\medskip

\section{Introduction}
\noindent This paper is devoted to the study of perturbations of the Crapper waves through gravity and concentrated vorticity. The problem that we analyze is the free-boundary stationary Euler equation with vorticity

\begin{subequations}\label{v-euler}
\begin{align}[left=\empheqlbrace\,]
&v\cdot\nabla v+\nabla p+ge_2=0\hspace{1cm}\textrm{in}\hspace{0.3cm}\Omega\label{v-euler1}\\[2mm]
&\nabla\cdot v=0\hspace{3.35cm}\textrm{in}\hspace{0.3cm}\Omega\label{v-euler2}\\[2mm]
&\nabla^{\perp}\cdot v=\omega \hspace{3.05cm}\textrm{in}\hspace{0.3cm}\Omega\label{v-euler3}\\[2mm]
&v\cdot n =0 \hspace{3.45cm}\textrm{on}\hspace{0.3cm} \mathcal{S}\label{v-euler4}\\[2mm]
&p=TK\hspace{3.5cm}\textrm{on}\hspace{0.3cm}\mathcal{S}\label{v-euler5}
\end{align}
\end{subequations}
\medskip

\noindent Here $v$ and $p$ are the velocity and the pressure, respectively; $g$  is the gravity, $e_2$ is the second vector of the Cartesian basis, $K$ is the curvature of the free boundary, $T$ the surface tension and $\omega$ the vorticity, that we will specify later. Moreover, since $\Omega$ is defined as a fluid region and  $\mathbb{R}^2\setminus \Omega$ as a vacuum region, there exists an interface $\mathcal{S}$ that separates the two regions and $n$ is the normal vector at the interface. We parametrize the interface with $z(\alpha)=(z_1(\alpha),z_2(\alpha))$, for $\alpha\in[-\pi,\pi]$. Thus $\Omega$ is defined for $-\pi< x< \pi$ and $y$ below the interface $z(\alpha)$ with $z_2(\pm \pi)=1$.\\

\noindent The Crapper waves are exact solutions of the water waves problem with surface tension at infinite depth. In \cite{Crapper1957}, Crapper proves the existence of pure capillary waves with an overhanging profile. Its result has been extended in \cite{Kinnersley1976} by Kinnersley for the finite depth case and in \cite{AAW2013} by Akers-Ambrose-Wright by adding a small gravity. In \cite{ASW2014} Ambrose-Strauss-Wright analyze the global bifurcation problem for traveling waves, considering the presence of two fluids and in \cite{CEG2016} and \cite{CEG2019}, C\'ordoba-Enciso-Grubic add beyond the small gravity a small density in the vacuum region in order to prove the existence of self-intersecting Crapper solutions with two fluids.\\
\noindent In the present paper we will deal with rotational waves. The literature about these waves is very recent and the first important result is the one by Constantin and Strauss \cite{CS2004}. They study the rotational gravity water waves problem without surface tension at finite depth and they are able to prove the existence of large amplitude waves. Later, in \cite{CV2011}, Constantin and Varvaruca extend the Babenko equation for irrotational flow \cite{Babenko87} to the gravity water waves with constant vorticity at finite depth. They remark that the new formulation opens the possibility of using global bifurcation theory to show the existence of large amplitude and possibly overhanging profiles. Furthermore, in a recent paper \cite{CSV2016}, the same authors construct waves of large amplitude via global bifurcation. Such waves could have overhanging profiles but their explicit existence is still an open problem.\\
\noindent Furthermore, there are some new results by Hur and Vanden-Broeck \cite{Hur-Vanden2020} and by Hur and Wheeler \cite{Hur-Wheeler2020}, where the authors prove the numerical and further analytical existence of a new exact solution for the periodic traveling waves in a constant vorticity flows of infinite depth, in the absence of gravity and surface tension. They show that the free surface is the same as that of Crapper's capillary waves in an irrotational flow.\\
\noindent Concerning the presence of surface tension in a rotational fluid we recall the works by Wahl\'en, in  \cite{Wahlen2006-1}, where the author proves the existence of symmetric regular capillary waves
for arbitrary vorticity distributions, provided that the wavelength is small enough and  in \cite{Wahlen2006-2}, he adds a gravity force acting at the interface and proves the existence of steady periodic capillary-gravity waves. As far as we know, there is not a proof of the existence of overhanging waves in both capillary and gravity-capillary rotational settings, with a fixed period.  In \cite{deBoeck2014}, De Boeck shows that Crapper waves are limiting configuration for both gravity-capillary water waves in infinte depth (see also \cite{AAW2013}) and gravity-capillary water waves with constant vorticity at finite depth. His formulation comes from the one introduced in \cite{CV2011} and the idea is based on taking a small period, which implies that Crapper's waves govern both gravity-capillary and gravity-capillary with constant vorticity at finite depth. 
\noindent  Differently from his work, we will consider a fixed period and small and concentrated vorticity as the point vortex and the vortex patch. \\
\noindent In \cite{SWZ2013}, Shatah, Walsh and Zheng study the capillary-gravity water waves with concentrated vorticity and they extend their work in \cite{EWZ2019} by considering an exponential localized vorticity; in both cases they perturb from the flat and they do not consider overhanging profiles.\\ 
\noindent  However, the technique we will use is completely different from the cited papers since we would like to show the existence of a perturbation of Crapper's waves with both small concentrated vorticity and small gravity. 

\subsection{Outline of the paper}
\noindent In section \ref{settings} we describe the setting in which we work and we introduce a new formulation for the problem \eqref{v-euler}, through the stream function and a proper change of coordinates to fix the domain. In section \ref{point-case} we describe the point vortex formulation and  the principal operators that identify our problem. In the end of the section we will prove the main theorem \ref{point-existence}, which shows the existence of a perturbation of Crapper's waves with a small point vortex. In the last section we introduce the problem \eqref{v-euler} with a vortex patch, which we identify through three operators and the implicit function theorem allow us to prove the existence of a perturbation of Crapper's waves also with a small vortex patch, theorem \ref{patch-existence}.
\bigskip

\section{Setting of the problem}\label{settings}
\noindent The interface $\mathcal{S}=\partial\Omega$, between the fluid region with density $\rho=1$ and the vacuum region, has a parametrization $z(\alpha)$ which satisfies  the periodicity conditions 

$$z_1(\alpha+2\pi)=z_1(\alpha) +2\pi, \hspace{1cm} z_2(\alpha+2\pi)=z_2(\alpha),$$
\medskip

\noindent and it is symmetric with respect to the $y-$axis

\begin{equation}\label{z-parity}
z_1(\alpha)=-z_1(-\alpha), \hspace{1cm} z_2(\alpha)=z_2(-\alpha).
\end{equation}
\medskip

\noindent The aim of this paper is to prove the existence of perturbations of the Crapper waves with vorticity  through the techniques developed in \cite{OS2001}, in \cite{AAW2013} and \cite{CEG2016}. First of all we will rewrite the system \eqref{v-euler} in terms of the stream function and then we will do some changes of variables in order to modify the fluid region and to analyse the problem in a more manageable domain. The key point is the use of the implicit function theorem to show that in a neighborhod of the Crapper solutions there exists a perturbation due to the presence of the gravity and the vorticity.

\subsection{The stream formulation with vorticity}
\noindent The fluid flow is governed by the incompressible stationary Euler equations \eqref{v-euler}. The incompressibility condition \eqref{v-euler2} implies the existence of a stream function $\psi:\Omega\rightarrow\mathbb{R}$, with $v=\nabla^{\perp}\psi$ and the kinematic boundary condition \eqref{v-euler4} implies $\psi=0$ on $\mathcal{S}$. In addition we can rewrite the equation \eqref{v-euler1} at the interface by using the condition  \eqref{v-euler5} and the fact that the vorticity we consider is concentrated in the domain $\Omega$, we end up in the Bernoulli equation.

\begin{equation}
\frac{1}{2} |v|^2+TK+gy=\textrm{constant}.
\end{equation}
\medskip

\noindent We  can write the system \eqref{v-euler} in terms of the stream function as follows

\begin{subequations}\label{psi-euler}
\begin{align}[left=\empheqlbrace\,]
&\Delta\psi=\omega\hspace{6cm}\textrm{in}\hspace{0.3cm}\Omega\label{psi-eul1}\\
&\psi=0\hspace{6.4cm}\textrm{on}\hspace{0.3cm}\mathcal{S}\label{psi-eul2}\\
&\frac{1}{2}|\nabla\psi|^2+gy+TK=\textrm{constant}\hspace{1.9cm}\textrm{on}\hspace{0.3cm}\mathcal{S}\label{psi-eul3}\\
&\frac{\partial\psi}{\partial x}=0 \hspace{6.1cm}\textrm{on}\hspace{0.3cm} x=\pm\pi\label{psi-eul4}\\
&\lim_{y\rightarrow 0}\left(\frac{\partial\psi}{\partial y},-\frac{\partial\psi}{\partial x}\right)=(c,0)\label{psi-eul5}
\end{align}
\end{subequations}
\medskip

\noindent where, the condition \eqref{psi-eul4} comes from the periodic and symmetric assumptions and the condition \eqref{psi-eul5} means that the flow becomes uniform at the infinite bottom and $c\in\mathbb{R}$ is the wave speed. The main problem we have to face is the absence of a potential and is due to the rotationality of the problem. We will treat the point vortex and the vortex patch in two different ways, since the singularity of the problem is distinct, but before dealing with our problem we will focus on the general framework.
\bigskip

\subsection{The general vorticity case}\label{subsec-change-variables}\label{general-vorticity}
\noindent  The main difficulties of the problem \eqref{psi-euler} are the presence of a moving interface and the absence of a potential, since the fluid is not irrotational. We recall the Zeidler theory \cite{Zeidler} about pseudo-potential, so we introduce the function $\phi$, which satisfies the following equations

\begin{equation}\label{pseudo-potential}
\left\{\begin{array}{lll}
\displaystyle\frac{\partial\phi}{\partial x}=W(x,y)\frac{\partial\psi}{\partial y}\\[2mm]
\displaystyle\frac{\partial\phi}{\partial y}=-W(x,y)\frac{\partial\psi}{\partial x},
\end{array}\right.
\end{equation}
\medskip

\noindent where $W(x,y)$ is exactly equal to $1$ when the fluid is irrotational and satisfies 

\begin{equation}\label{x-y-W}
\frac{\partial W}{\partial x}\frac{\partial\psi}{\partial x}+\frac{\partial W}{\partial y}\frac{\partial\psi}{\partial y}+W\Delta\psi=0.
\end{equation}
\medskip

\noindent We transform the problem from the $(x,y)$-plane into the $(\phi,\psi)$-plane, by taking the advantage of the fact that the stream function is zero at the interface, see fig. \ref{omega-to-disk} . Furthermore, we consider the case of symmetric waves, then it follows that 

\begin{equation}\label{symmetry}
\left\{\begin{array}{lll}
\phi(x,y)=-\phi(-x,y)\\[3mm]
\psi(x,y)=\psi(-x,y),
\end{array}\right.
\end{equation}
\medskip

 \noindent and they satisfy the following relations, coming from \eqref{pseudo-potential}.

\begin{equation}\label{x-y-Jacobian}
\begin{pmatrix}
\displaystyle\frac{\partial x}{\partial\phi} & \displaystyle\frac{\partial x}{\partial\psi}\\
\displaystyle\frac{\partial y}{\partial\phi} & \displaystyle\frac{\partial y}{\partial\psi}\\
\end{pmatrix}
=\frac{1}{W(v_1^2+v_2^2)} 
\begin{pmatrix}
v_1 & -W v_2\\
v_2 & W v_1
\end{pmatrix},
\end{equation} 
\medskip

\noindent where $v_1, v_2$ are the components of the velocity field. Moreover, we want to write the system in a non-dimensional setting, thus the new variables are 
$$(\phi,\psi)=\frac{1}{c}(\phi^*,\psi^*), \quad
(v_1,v_2)=\frac{1}{c} (v_1^*,v_2^*), \quad \omega=\frac{1}{c}\omega^*,$$

\noindent where the variables with the star are the dimensional one and $c$ is the wave speed. The properties of our problem allow us to pass from $\Omega$ into $\tilde{\Omega}$, defined as follows

\begin{equation}\label{Omega-tilde}
\tilde{\Omega}:=\{(\phi,\psi): -\pi<\phi<\pi, -\infty<\psi<0\}.
\end{equation}
\medskip

\noindent We have to transform the system \eqref{psi-euler} and the equation \eqref{x-y-W} in the new coordinates. So we take the derivative with respect to $\phi$ of the condition \eqref{psi-eul3} and we get

\begin{equation}\label{phi-psi-DBernoulli}
\frac{\partial}{\partial\phi}\left(\frac{v_1^2+v_2^2}{2}\right)+p\frac{v_2}{v_1^2+v_2^2}-q\frac{\partial}{\partial\phi}\left[\frac{W}{\sqrt{v_1^2+v_2^2}}\left(v_1\frac{\partial v_2}{\partial\phi}-v_2\frac{\partial v_1}{\partial\phi}\right)\right]=0,
\end{equation}
\medskip

\noindent where $\displaystyle p=\frac{g}{c^2}$ and $\displaystyle q=\frac{T}{c^2}$ and \eqref{x-y-W} becomes

\begin{equation}\label{phi-psi-W}
(v_1^2+v_2^2)\frac{\partial W}{\partial\psi}=W\omega.
\end{equation}
\medskip

\noindent The problem we study is periodic, so it is more natural to do the analysis in a circular domain. We introduce the independent variable $\zeta=e^{-i\phi+\psi}$, where $\phi+i\psi$ runs in $\tilde{\Omega}$ and $\zeta$ in the unit disk, so $\zeta=\rho e^{i\alpha}$.  The relation between $(\phi,\psi)$ and the variable in the disk $(\alpha,\rho)$ is the following $(\phi,\psi)=(-\alpha,\log(\rho))$, where $-\pi<\alpha<\pi$ and $0 <\rho<1$. Thus, we pass from $\tilde{\Omega}$ into  the unit disk, see fig. \ref{omega-to-disk}. 

\begin{figure}[htbp]
\centering
\includegraphics[scale=0.5]{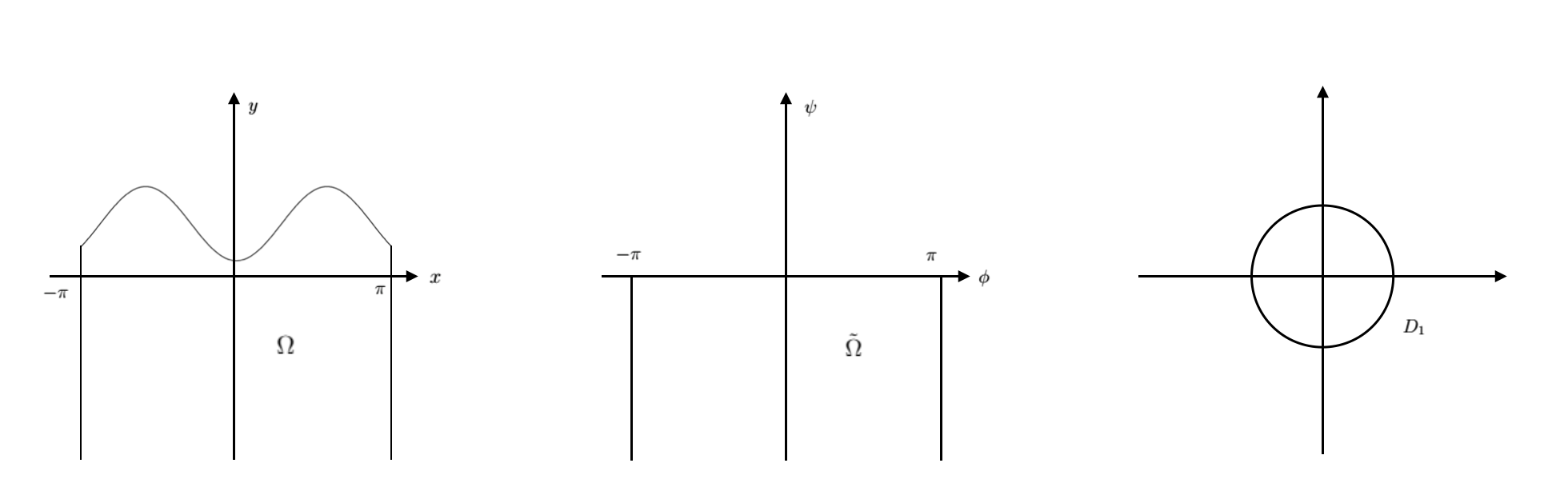}
\caption{The domains $\Omega$, $\tilde{\Omega}$ and the disk.}\label{omega-to-disk}
\end{figure}
\medskip

\noindent Furthermore we define the dependent variables $\tau(\alpha,\rho)$ and $\theta(\alpha,\rho)$ as follows

\begin{equation}\label{tau-theta}
\tau=\frac{1}{2}\log(v_1^2+v_2^2),\quad \theta=\arctan\left(\frac{v_2}{v_1}\right)
\end{equation}
\medskip

\noindent Thanks to \eqref{tau-theta}, the equation \eqref{phi-psi-W} for $W$ becomes

\begin{equation*}
e^{2\tau}\rho\frac{\partial W}{\partial\rho}=W\omega,
\end{equation*}
\medskip

\noindent then we have 

\begin{equation}\label{alpha-rho-W}
\displaystyle W(\alpha,\rho)=\exp\left(\int_{0}^{\rho}\omega\frac{e^{-2\tau(\alpha,\rho')}}{\rho'}\,d\rho'\right). 
\end{equation}
\medskip

\noindent The derivative of the Bernoulli equation \eqref{phi-psi-DBernoulli}, computed at the interface $z(\alpha)$ which corresponds to $\rho=1$, becomes

\begin{equation}\label{alpha-rho-DBernoulli}
\frac{\partial}{\partial\alpha}\left(\frac{1}{2}e^{2\tau(\alpha,1)}\right)-p\frac{e^{-\tau(\alpha,1)}\sin(\theta(\alpha,1))}{W(\alpha,1)}+q\frac{\partial}{\partial\alpha}\left(W(\alpha,1)e^{\tau(\alpha,1)}\frac{\partial\theta}{\partial\alpha}\right)=0.
\end{equation}
\bigskip

\section{The point vortex case}\label{point-case}
\subsection{The point vortex framework}
 We consider a point of constant vorticity, that does not touch the interface $z(\alpha)$, defined as $\omega=\omega_0 \delta((x,y)-(0,0))$, where $\delta((x,y)-(0,0))$ is a delta distribution taking value at the point $(0,0)$ and $\omega_0$ is a small constant. In addition, since we have a fluid with density $1$ inside the domain $\Omega$ and the vacuum in $\mathbb{R}^2\setminus\Omega$, then there is a discontinuity of the velocity field at the interface and a concentration of vorticity $\tilde{\omega}(\alpha)\delta((x,y)-(z_1(\alpha),z_2(\alpha)))$, where $\tilde{\omega}(\alpha)$ is the amplitude of the vorticity along the interface. This implies the stream function $\psi$  in $\Omega$ to be the sum of an harmonic part 

\begin{equation}\label{harmonic-point-stream}
\psi_H(x,y)=\frac{1}{2\pi}\int_{-\pi}^{\pi}\log\left|(x,y)-(z_1(\alpha'), z_2(\alpha'))\right|\tilde{\omega}(\alpha')\,d\alpha',
\end{equation}

\noindent which is continuous over the interface and another part related to the point vortex. The velocity can be obtained by taking the orthogonal gradient of the stream function and we have\\

\begin{equation}\label{velocity-point}
v(x,y)=(\partial_y\psi_{H}(x,y), -\partial_x\psi_H(x,y))+\frac{\omega_0}{2\pi}\frac{(y,-x)}{x^2+y^2}.
\end{equation}
\medskip

\noindent However, in order to describe the point vortex problem we have to adapt the kinematic boundary condition \eqref{v-euler4} and the Bernoulli equation \eqref{v-euler5}, equivalent to \eqref{psi-eul3}. At first, let us compute the velocity at the interface by taking the limit in the normal direction and we get

\begin{equation}
\begin{split}
v(z(\alpha))&=(\partial_{z_2}\psi_H,-\partial_{z_1}\psi_H)+\frac{1}{2}\frac{\tilde{\omega}(\alpha)}{|\partial_{\alpha}z|^2}\partial_{\alpha}z+\frac{\omega_0}{2\pi}\frac{(z_2(\alpha),-z_1(\alpha))}{|z(\alpha)|^2}\\[3mm]
&=BR(z(\alpha),\tilde{\omega}(\alpha))+\frac{1}{2}\frac{\tilde{\omega}(\alpha)}{|\partial_{\alpha}z|^2}\partial_{\alpha}z+\frac{\omega_0}{2\pi}\frac{(z_2(\alpha),-z_1(\alpha))}{|z(\alpha)|^2},
\end{split}
\end{equation}
\medskip

\noindent where $BR(z(\alpha),\tilde{\omega}(\alpha))$ is  the Birkhoff-Rott integral

\begin{equation*}
BR(z(\alpha),\tilde{\omega}(\alpha))=\frac{1}{2\pi}\textrm{P.V.}\int_{-\pi}^{\pi}\frac{(z(\alpha)-z(\alpha'))^{\perp}}{|z(\alpha)-z(\alpha')|^2}\cdot\tilde{\omega}(\alpha')\,d\alpha'.
\end{equation*}
\medskip

\noindent Thus the condition \eqref{v-euler4} becomes

\begin{equation}\label{point-kbc}
\begin{split}
&v(z(\alpha))\cdot(\partial_{\alpha}z(\alpha))^{\perp}=\left(BR(z(\alpha),\tilde{\omega}(\alpha))+\frac{\omega_0}{2\pi}\frac{(z_2(\alpha),-z_1(\alpha)}{|z(\alpha)|^2}\right)\cdot \partial_{\alpha}z(\alpha)^{\perp}=0.
\end{split}
\end{equation}
\medskip

\noindent To deal with the Bernoulli equation and to reach a manageable formulation, we have to use the change of variables described in section \ref{general-vorticity}. We pass from the domain $\Omega$ in $(x,y)$ variables, fig. \ref{omega-to-disk} into $\tilde{\Omega}$ in $(\phi, \psi)$ and finally in the unit disk. In order to pass from $\Omega$ into $\tilde{\Omega}$ we use the pseudo-potential defined in \eqref{pseudo-potential} and \eqref{x-y-W}. Moreover, as one can see in fig. \ref{disk} (left), the interface $z(\alpha)$ is sending in the line $\psi=0$, thanks to condition \eqref{psi-eul2}, and the point vortex is still a point $(0,\psi_0)$ on the vertical axis due to the oddness of $\phi$. In order to pass from $\tilde{\Omega}$ into the unit disk, see fig. \ref{disk} (right), we use the function $e^{\psi-i\phi}=\rho e^{i\alpha}$ and the point vortex $(0,\psi_0)$ becomes a point $(0,\rho_0)$, it does not depend on the angle $\alpha$.
\medskip

\begin{figure}[htbp]
\centering
\includegraphics[scale=0.5]{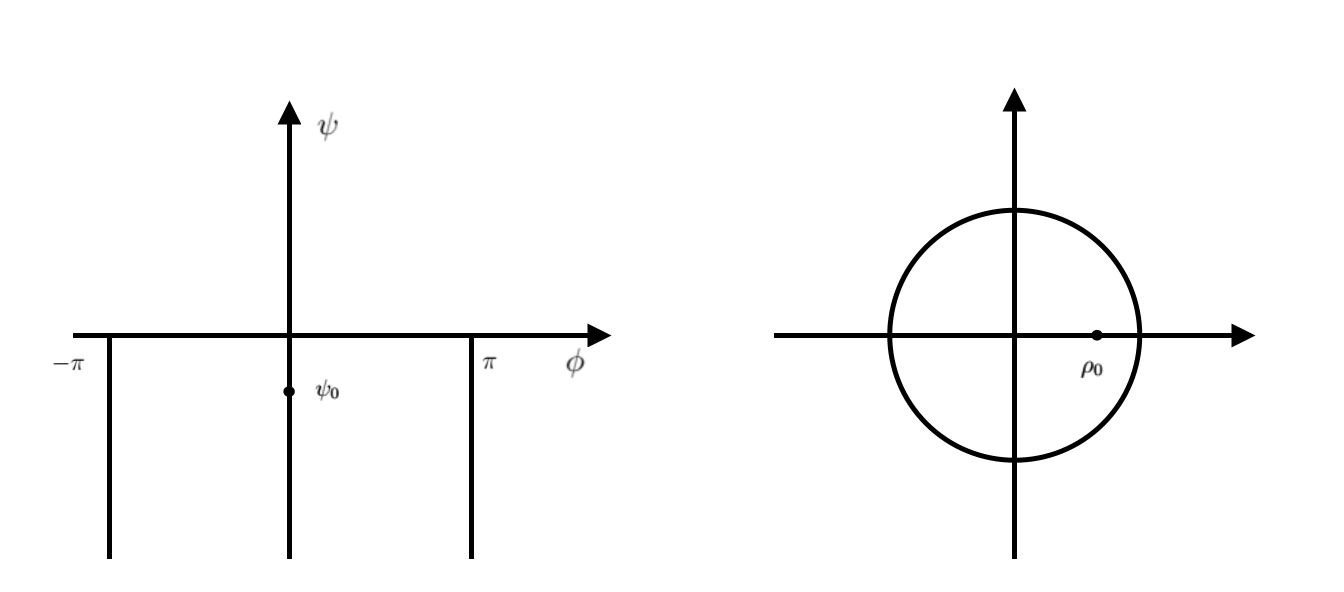}\label{disk}
\caption{The domain $\tilde{\Omega}$ and $D_1$}\label{disk}
\end{figure}
\medskip

\noindent After this change of variables, we rewrite the equation \eqref{alpha-rho-W} for $W(\alpha,\rho)$, by substituting $\omega=\omega_0\delta((\alpha,\rho)-(0,\rho_0))$. And we have

\begin{equation}\label{W_0}
W(\alpha,\rho)=
\left\{
\begin{array}{rl}
1 & \alpha\neq 0\\[2mm]
\displaystyle \exp\left(\frac{\omega_0 e^{-2\tau(\alpha,\rho_0)}}{\rho_0}\right) & \alpha=0,\hspace{0.3cm}\rho_0\in (0,\rho).
\end{array}\right.
\end{equation}
\medskip

\noindent We immediately point out that in this case the function $W(\alpha,\rho)=W_{\omega_0,\rho_0}=W_0\in\mathbb{R}$ and the constant is exactly one when there is no vorticity.\\
\noindent The derivative of the Bernoulli equation \eqref{alpha-rho-DBernoulli} becomes

\begin{equation}\label{point-DBernoulli}
\frac{\partial}{\partial\alpha}\left(\frac{1}{2}e^{2\tau(\alpha,1)}\right)-p\frac{e^{-\tau(\alpha,1)}\sin(\theta(\alpha,1))}{W_0}+q W_0\frac{\partial}{\partial\alpha}\left(e^{\tau(\alpha,1)}\frac{\partial\theta}{\partial\alpha}\right)=0
\end{equation}
\medskip

\noindent By integrating with respect to $\alpha$, we get

\begin{equation*}
\frac{1}{2}e^{2\tau(\alpha,1)}-\frac{p}{W_0}\int_{-\pi}^{\alpha} e^{-\tau(\alpha',1)}\sin(\theta(\alpha',1))\,d\alpha'+q W_0 e^{\tau(\alpha,1)}\frac{\partial\theta}{\partial\alpha}(\alpha,1)=\tilde{\gamma}.
\end{equation*}
\medskip

\noindent In the pure capillarity case the constant is exactly $\frac{1}{2}$. For this reason we take $\tilde{\gamma}=\frac{1}{2}+B$, where $B$ is a perturbation of the Crapper constant, see \cite{OS2001}. We multiply the equation  by $e^{-\tau(\alpha,1)}$ and we get a new formulation for the Bernoulli equation.

\begin{equation}\label{point-Bernoulli-tau-theta}
\begin{split}
&\sinh(\tau(\alpha,1))-\frac{p}{W_0} e^{-\tau(\alpha,1)}\left(\int_{-\pi}^{\alpha}e^{-\tau(\alpha',1)}\sin(\theta(\alpha',1)\,d\alpha'{-1}\right)\\[3mm]
&+q W_0\frac{\partial\theta(\alpha,1)}{\partial\alpha}-Be^{-\tau(\alpha,1)}=0.
\end{split}
\end{equation}
\medskip

\noindent  We can solve our problem by finding $2\pi$ periodic functions $\tau(\alpha)$ even and $\theta(\alpha)$ odd, a function $\tilde{\omega}(\alpha)$ even, that satisfy the following equations \eqref{point-kbc} and \eqref{point-Bernoulli-tau-theta}. However, in subsection \ref{subsec-change-variables}, we explain the necessary  change of variables to fix the domain. We observe that at the interface  $z(\alpha)$, we have  $\psi(z(\alpha))=0$ and $\rho=1$, respectively. Thus,


\begin{equation}\label{phi-parametrization}
\phi(z(\alpha))=-\alpha\quad\Longrightarrow\quad \nabla\phi(z(\alpha))\cdot\partial_{\alpha}z(\alpha)= -1,
\end{equation}
\medskip

\noindent from \eqref{pseudo-potential} we get that \eqref{phi-parametrization}, can be written as follows

\begin{equation}\label{v-tangential}
W_0 v(z(\alpha))\cdot\partial_{\alpha}z(\alpha)=-1.
\end{equation}
\medskip

\noindent Since the equation \eqref{v-tangential} has been obtained by using the kinematic boundary condition \eqref{point-kbc}, then to solve our problem we will use the Bernoulli equation \eqref{point-Bernoulli-tau-theta} and the equation \eqref{v-tangential}.
\bigskip

\subsection{Crapper formulation}\label{Crapper formulation}
\noindent Our goal is to prove the existence of overhanging waves with the presence of concentrated vorticity, such as a point vortex or a vortex patch (Section \ref{patch}). It is well-known that without vorticity ($\omega_0=0$ equivalent to $W_0=1$), in \cite{AAW2013} the authors prove the existence of gravity-capillary overhanging waves. If we remove also the gravity then there is the pillar result of Crapper \cite{Crapper1957}, where the problem was to find a $2\pi$ periodic, analytic function $f_c=\theta_c+i\tau_c$ in the lower half plane which solves the Bernoulli equation 

\begin{equation}\label{Crapper-Bernoulli}
\sinh(\tau_c)+q \frac{\partial\theta_c}{\partial\alpha}=0,
\end{equation}
\medskip

\noindent where $q=\frac{T}{c^2}$. Furthermore, the analyticity of the function $f$ implies that $\tau_c$ can be written as the Hilbert transform of $\theta_c$ at the boundary $\rho=1$, so the equation above reduces to an equation in the variable $\theta_c$,

\begin{equation}\label{Crapper-Bernoulli-only-theta}
\sinh(\mathcal{H}\theta_c)+q \frac{\partial\theta_c}{\partial\alpha}=0.
\end{equation}
\medskip

\noindent This problem admits a family of exact solutions,

\begin{equation}\label{Crapper-solutions}
f_c(w)=2i\log\left(\frac{1+Ae^{-iw}}{1-Ae^{-iw}}\right),
\end{equation}
\medskip

\noindent  where $w=\phi+i\psi$ and in this case $(\phi,\psi)$ are harmonic conjugates. The parameter $A$ is defined in $(-1,1)$ and for $|A|<A_0=0.45467\ldots$, the interface do not have self-intersection. Moreover, by substituting \eqref{Crapper-solutions} into \eqref{Crapper-Bernoulli-only-theta} for $\rho=1$, we get  $q=\frac{1+A^2}{1-A^2}$. This implies 

\begin{equation}\label{surface-tension}
T=\frac{1+A^2}{1-A^2}c^2.
\end{equation}
\medskip

\noindent By using \eqref{x-y-Jacobian} in the Crapper case so with $W=1$, coupled with $\phi=-\alpha$ and $\rho=1$,  we get 

\begin{equation}\label{Crapper-z-derivative}
\partial_{\alpha}z^c(\alpha)=- e^{-\tau_c(\alpha)+i\theta_c(\alpha)}.
\end{equation}
\medskip

\noindent We focus on this kind of waves because for some values of the parameter $A$, these waves are overhanging.

\bigskip

\subsection{Perturbation of Crapper waves with a point of vorticity}

\noindent In our formulation, the main difference with respect to the Crapper \cite{Crapper1957} waves is in the function $f=\tau+i\theta$ which is not analytic because of the presence of vorticity. The idea is to prove that our solutions are perturbation of the Crapper waves. If we recall the Crapper solution with small gravity but without vorticity, $(\theta_A,\tau_A)$, we know that $f_A=\theta_A+i\tau_A$ is now  analytic and $(\theta_A,\tau_A)$ satisfy the following relations in both $(\phi,\psi)$ and $(\alpha,\rho)$ variables.

\begin{equation} \label{tauCR-thetaCR}
\begin{cases}
\displaystyle\frac{\partial\theta_A}{\partial\phi}=\frac{\partial\tau_A}{\partial\psi} \\[4mm]
\displaystyle\frac{\partial\theta_A}{\partial\psi}=-\frac{\partial\tau_A}{\partial\phi}
\end{cases}\Longrightarrow
\begin{cases}
\displaystyle\frac{\partial\theta_A}{\partial\alpha}=-\rho\frac{\partial\tau_A}{\partial\rho} \\[4mm]
\displaystyle\rho\frac{\partial\theta_A}{\partial\rho}=\frac{\partial\tau_A}{\partial\alpha}.
\end{cases}
\end{equation}
\medskip

\noindent Moreover, $\tau_A=\mathcal{H}\theta_A$ at the interface. The idea is to write our dependent variables $\tau$ and $\theta$ as the sum of a Crapper part and a small perturbation, due to the small vorticity. So we have

\begin{equation}\label{point-tau-theta-perturbations}
\tau=\tau_A+\omega_0\tilde{\tau},\quad \theta=\theta_A+\omega_0\tilde{\theta}.
\end{equation}
\medskip

\noindent So the Bernoulli equation \eqref{point-Bernoulli-tau-theta}, reduces

\begin{equation}\label{point-Bernoulli}
\begin{split}
&\sinh(\mathcal{H}\theta_A+\omega_0\tilde{\tau})-p e^{-\mathcal{H}\theta_A-\omega\tilde{\tau}}\left(\frac{1}
{W_0}\int_{-\pi}^{\alpha}e^{-\mathcal{H}\theta_A-\omega_0\tilde{\tau}}\sin(\theta_A+\omega_0\tilde{\theta})\,d\alpha'{-1}\right)\\[2mm]
&+q\frac{\partial(\theta_A+\omega_0\tilde{\theta})}{\partial\alpha}W_0-Be^{-\mathcal{H}\theta_A-\omega_0\tilde{\tau}}=0\hspace{1cm}\textrm{at}\hspace{0.3cm}\rho=1.
\end{split}
\end{equation}
\medskip

\noindent However, we will figure out that $\tilde{\tau}$ and $\tilde{\theta}$ are functions of $\theta_A$ and so \eqref{point-Bernoulli} will be an equation in the variable $\theta_A$. In order to end up with this statement we need to use some properties of our problem.  We use the incompressibility and rotational conditions and we get the following relations for $(\tau,\theta)$

\begin{equation}\label{point-theta-tau-quasi-Cauchy-Riemann-1}
\left\{\begin{array}{lll}
\displaystyle\frac{\partial\theta}{\partial\psi}=-W_0\frac{\partial\tau}{\partial\phi}\\[5mm]
\displaystyle\frac{\partial\theta}{\partial\phi}=\frac{\omega_0 e^{-2\tau(0,\psi_0)}\delta((\phi,\psi)-(0,\psi_0))}{W_0}+\frac{1}{W_0}\frac{\partial\tau}{\partial\psi}
\end{array}\right.
\end{equation}
\medskip

%

\noindent By substituting \eqref{point-tau-theta-perturbations} in \eqref{point-theta-tau-quasi-Cauchy-Riemann-1}, we get

\begin{equation}\label{point-tilde-tau-theta-quasi-Cauchy-Riemann-1}
\left\{\begin{array}{lll}
\displaystyle\omega_0\frac{\partial\tilde{\theta}}{\partial\psi}=-W_0\omega_0\frac{\partial\tilde{\tau}}{\partial\phi}-W_0\frac{\partial\tau_A}{\partial\phi}-\frac{\partial\theta_A}{\partial\psi}\\[5mm]
\displaystyle\omega_0\frac{\partial\tilde{\theta}}{\partial\phi}=\frac{\omega_0 e^{-2\mathcal{H}\theta_A(0,\psi_0)-2\omega_0\tilde{\tau}(0,\psi_0)}\delta((\phi,\psi)-(0,\psi_0))}{W_0}+\frac{1}{W_0}\left(\omega_0\frac{\partial\tilde{\tau}}{\partial\psi}\frac{\partial\tau_A}{\partial\psi}\right)+\frac{\partial\theta_A}{\partial\phi}
\end{array}\right.
\end{equation}
\medskip


\noindent If we cross systems \eqref{tauCR-thetaCR} with system \eqref{point-tilde-tau-theta-quasi-Cauchy-Riemann-1} then we obtain

\begin{equation}\label{point-tilde-theta-phi-psi}
\left\{\begin{array}{lll}
\displaystyle\omega_0\frac{\partial\tilde{\theta}}{\partial\psi}=-W_0\omega_0\frac{\partial\tilde{\tau}}{\partial\phi}+(W_0-1)\frac{\partial\theta_A}{\partial\psi}\\[5mm]
\displaystyle\omega_0\frac{\partial\tilde{\theta}}{\partial\phi}=\frac{\omega_0 e^{-2\mathcal{H}\theta_A(0,\psi_0)-2\omega_0\tilde{\tau}(0,\psi_0)}\delta((\phi,\psi)-(0,\psi_0))}{W_0}+\frac{}{W_0}\frac{\partial\tilde{\tau}}{\partial\psi}+\left(1+\frac{1}{W_0}\right)\frac{\partial\theta_A}{\partial\phi}
\end{array}\right.
\end{equation}
\medskip

\noindent By taking the derivative with respect to $\phi$ in the first equation and the derivative with respect to $\psi$ in the second equation and then the difference we get an elliptic equation

\begin{equation}\label{point-elliptic-phi-psi}
\displaystyle W_0\tilde{\varepsilon}\frac{\partial^2\tilde{\tau}}{\partial\phi^2}+\frac{1}{W_0}\tilde{\varepsilon}\frac{\partial^2\tilde{\tau}}{\partial\psi^2}+\frac{\omega_0 e^{-2\tau(0,\psi_0)}}{W_0}\frac{\partial}{\partial\psi}\delta(\phi,\psi-\psi_0) +\left(\frac{1-W_0^2}{W_0}\right)\frac{\partial^2\theta_A}{\partial\phi\partial\psi}=0
\end{equation}
\medskip

%
%
\noindent We can do the same as \eqref{point-theta-tau-quasi-Cauchy-Riemann-1}, \eqref{point-tilde-tau-theta-quasi-Cauchy-Riemann-1} and \eqref{point-tilde-theta-phi-psi} also in the variables $(\alpha,\rho)$ and the elliptic equation is the following

\begin{equation*}\label{point-elliptic-alpha-rho}
\begin{split}
&\frac{\rho}{W_0}\omega_0\frac{\partial^2\tilde{\tau}}{\partial\rho^2}+\frac{1}{W_0}\omega_0\frac{\partial\tilde{\tau}}{\partial\rho}+\frac{W_0}{\rho}\omega_0\frac{\partial^2\tilde{\tau}}{\partial\alpha^2}+\frac{\partial}{\partial\rho}\left(\frac{\omega_0 e^{-2\mathcal{H}\theta_A-2\omega_0\tilde{\tau}}}{W_0}\right)+\left(W_0-\frac{1}{W_0}\right)\frac{\partial^2\theta_A}{\partial\alpha\partial\rho}=0.
\end{split}
\end{equation*}
\medskip

\noindent Once we solve the elliptic equation we have a solution $\tilde{\tau}$ as a function of $\theta_A$ and thanks to the relations \eqref{point-tilde-theta-phi-psi} also $\tilde{\theta}$ is a function of $\theta_A$.

\subsection{The elliptic problem}
\noindent In this section we want to show how to solve the elliptic problem. For simplicity, we will study the problem in the $(\phi,\psi)$ coordinates thus, from \eqref{point-elliptic-phi-psi}, the system is

\begin{equation*}\label{point-elliptic-phi-psi-system}
\displaystyle W_0\omega_0\frac{\partial^2\tilde{\tau}}{\partial\phi^2}+\frac{1}{W_0}\omega_0\frac{\partial^2\tilde{\tau}}{\partial\psi^2}+\frac{\omega_0 e^{-2\tau(0,\psi_0)}}{W_0}\frac{\partial}{\partial\psi}\delta(\phi,\psi-\psi_0) +\left(\frac{1-W_0^2}{W_0}\right)\frac{\partial^2\theta_A}{\partial\phi\partial\psi}=0.
\end{equation*}
\medskip

\noindent The equation above is a linear elliptic equation with constant coefficients $\displaystyle W_0, \frac{1}{W_0}$. If we do a change of variables we obtain a Poisson equation. In the specific if we define $\phi=W_0\phi'$, then we have

\begin{equation*}
\begin{split}
&\frac{\partial f}{\partial\phi'}(\phi',\psi)=\frac{\partial f}{\partial\phi}\frac{\partial\phi}{\partial\phi'}=W_0\frac{\partial f}{\partial\phi}\Longrightarrow\frac{\partial f}{\partial\phi}=\frac{1}{W_0}\frac{\partial f}{\partial\phi'}\\[3mm]
&\frac{\partial^2 f}{\partial\phi^2}=\frac{1}{W_0^2}\frac{\partial^2 f}{\partial\phi'^2}.
\end{split}
\end{equation*}
\medskip

\noindent And the domain $\tilde{\Omega'}=\{(\phi',\psi):-\frac{\pi}{W_0}<\phi<\frac{\pi}{W_0}, -\infty<\psi<0\}$. By substituting in \eqref{point-elliptic-phi-psi}, we have

\begin{equation}\label{point-elliptic-phi-psi-NEW}
\displaystyle \omega_0\frac{\partial^2\tilde{\tau}}{\partial\phi'^2}+\omega_0\frac{\partial^2\tilde{\tau}}{\partial\psi^2}+\omega_0 e^{-2\tau(0,\psi_0)}\frac{\partial}{\partial\psi}\delta(\phi,\psi-\psi_0) +\left(\frac{1-W_0^2}{W_0}\right)\frac{\partial^2\theta_A}{\partial\phi'\partial\psi}=0
\end{equation}
\medskip

\noindent Since we are looking for $\tau\in H^2$ and we know that $\tilde{\tau}\in H^2$ so that its Laplacian is in $L^2(\tilde{\Omega'})$; then by the elliptic theory there exists a weak solution and  so we can invert the Laplace operator, \cite[Theorem 9.25]{Brezis}. We have

\begin{equation}\label{point-tilde-tau}
\omega_0\tilde{\tau}=\left(-\omega_0 e^{-2\tau(0,\psi_0)}\frac{\partial}{\partial\psi}\delta(\phi,\psi-\psi_0) -\left(\frac{1-W_0^2}{W_0}\right)\frac{\partial^2\theta_A}{\partial\phi'\partial\psi}\right)*G_2(\phi',\psi),
\end{equation}
\medskip

\noindent where $G_2$ is the Green function of the Poisson equation in $\tilde{\Omega'}$. 

\subsection{Existence of gravity rotational perturbed Crapper waves}\label{point-existence-Crapper}

\noindent The main theorem we want to prove is the following 

\begin{theorem}\label{point-existence}
Let us consider the water waves problem \eqref{v-euler}, with a small point vortex and a small gravity $g$. Then, for some values of $A<A_0$, defined in  \eqref{Crapper-solutions}, there exist periodic solutions to \eqref{v-euler} with 
overhanging profile.
\end{theorem}

\noindent In order to prove the existence of perturbed rotational Crapper waves we will apply the implitic function theorem around the Crapper solutions. 

\begin{theorem}[Implicit function theorem]\label{IFT}
Let $X, Y, Z$ be Banach spaces and $\zeta:X\times Y\rightarrow Z$ is a $C^k$, with $k\geq 1$. If $\zeta(x_*, y_*)=0$ and $D_x\zeta(x_*, y_*)$ is a bijection from $X$ to $Z$, then there exists $\varepsilon>0$ and a unique $C^k$ map $\chi:Y\rightarrow X$ such that $\chi(y_*)=x_*$ and $\zeta(\chi(y_*), y_*)=0$ when $\|y-y_*\|_Y\leq\varepsilon$.
\end{theorem}
\medskip

\noindent The operators that identify the water waves problem with a point vortex are the following 

\begin{subequations}
\label{point-water-waves}
\begin{align}
\begin{split}
&\mathcal{F}_1(\theta_A,\tilde{\omega}; B,p,\omega_0):=\sinh(\mathcal{H}\theta_A+\omega_0\tilde{\tau}(\theta_A))\\[2mm]
&\hspace{2cm}-p e^{-\mathcal{H}\theta_A-\omega_0\tilde{\tau}(\theta_A)}\left(\frac{1}{W_0}\int_{-\pi}^{\alpha}e^{-\mathcal{H}\theta_A-\omega_0\tilde{\tau}(\theta_A)}\sin(\theta_A+\omega_0\tilde{\theta}(\theta_A))\,d\alpha'{-1}\right)\\[2mm]
&\hspace{2cm}+q\frac{\partial(\theta_A+\omega_0\tilde{\theta}(\theta_A))}{\partial\alpha}W_0-Be^{-\mathcal{H}\theta_A-\omega_0\tilde{\tau}(\theta_A)}
\end{split}\label{point-water-waves-1}\\[5mm]
\begin{split}
&\mathcal{F}_2(\theta_A,\tilde{\omega};B, p,\omega_0):= W_0\left(2BR(z(\alpha),\tilde{\omega}(\alpha))\cdot\partial_{\alpha}z(\alpha)+\tilde{\omega}(\alpha)\right.\\[2mm]
&\hspace{2cm}\left.+\frac{\omega_0}{\pi}\frac{(z_2(\alpha),-z_1(\alpha))}{|z(\alpha)|^2}\cdot\partial_{\alpha}z(\alpha)\right)+2
\end{split}\label{point-water-waves-2}
\end{align}
\end{subequations}
\medskip

\noindent We have that 

$$(\mathcal{F}_1,\mathcal{F}_2)(\theta_A,\tilde{\omega};B, p,\omega_0): H^{2}_{odd}\times H^{1}_{even}\times\mathbb{R}^{{3}}\rightarrow H^{1}_{even}\times H^{1}_{even}.$$
\medskip

\subsubsection{Proof of Theorem \ref{point-existence}}\label{point-existence-proof}
We have to analyze the two operators. First we have to show that the operators are zero when computed at the point $(\theta_c,\tilde{\omega}_c;0,0,0)$.

\begin{equation}\label{F1}
\mathcal{F}_1(\theta_c,\tilde{\omega}_c;0,0,0)=\sinh(\mathcal{H}\theta_c)+q\frac{\partial\theta_c}{\partial\alpha}=0,
\end{equation}

\noindent since this is exactly \eqref{Crapper-Bernoulli-only-theta}.
\medskip

\noindent The second operator related to the kinematic boundary conditions satisfies

\begin{equation}\label{F2}
\mathcal{F}_2(\theta_c,\tilde{\omega}_c;0,0,0)=
2BR(z^c({\alpha}),\tilde{\omega}_c(\alpha))+\tilde{\omega}_c(\alpha)+2=0,
\end{equation}
\medskip

\noindent where $z^c(\alpha)$ is the parametrization of the Crapper interface and it is zero by construction \eqref{phi-parametrization}. 
\medskip

\noindent Now, we compute all the Fr\'echet derivatives. We will take the derivatives with respect to $\theta_A$ and $\tilde{\omega}$, then we will compute them at the point $(\theta_c,\tilde{\omega}_c;0,0,0)$ and we will show their invertibility. For the operator $\mathcal{F}_1$ we observe that 

$$D_{\tilde{\omega}}\mathcal{F}_1(\theta_c,\tilde{\omega}_c,;0,0,0)=0.$$
\medskip
 
\noindent It remains to compute the derivative with respect to $\theta_A$.

\begin{equation*}
\begin{split}
&D_{\theta_A}\mathcal{F}_1=\left[\frac{d}{d\mu}\mathcal{F}_1(\theta_A+\mu\theta_1,\psi_H,c;B,p,\varepsilon,\omega_0)\right]_{|\mu=0}=\left[\frac{d}{d\mu}\left[\sinh(\mathcal{H}\theta_A+\mu\mathcal{H}\theta_1+\omega_0\tilde{\tau}(\theta_A+\mu\theta_1))\right.\right.\\[3mm]
&-p e^{-\mathcal{H}\theta_A-\mu\mathcal{H}\theta_1+\omega_0\tilde{\tau}(\theta_A+\mu\theta_1)}\left(\frac{1}{W_0}\int_{-\pi}^{\alpha}e^{-\mathcal{H}\theta_A+\mu\mathcal{H}\theta_1+\omega_0\tilde{\tau}(\theta_A+\mu\theta_1)}\sin(\theta_A+\mu\theta_1+\omega_0
\tilde{\theta})\,d\alpha'{-1}\right)\\[2mm]
&\left.\left.+q\frac{\partial(\theta_A+\mu\theta_1+\omega_0\tilde{\theta})}{\partial\alpha}W_0-Be^{-\mathcal{H}\theta_A-\mu\mathcal{H}\theta_1+\omega_0\tilde{\tau}(\theta_A+\mu\theta_1)}\right]\right]_{|\mu=0}=\\[3mm]
&=\cosh(\mathcal{H}\theta_A+\omega_0\tilde{\tau}(\theta_A))\left(\mathcal{H}\theta_1+\omega_0\left[\frac{d}{d\mu}\tilde{\tau}(\theta_A+\mu\theta_1)\right]_{|\mu=0}\right)\\[3mm]
&-p\left[\frac{d}{d\mu}\left[e^{-\mathcal{H}\theta_A-\mu\mathcal{H}\theta_1+\omega_0\tilde{\tau}(\theta_A+\mu\theta_1)}\cdot\right.\right.\\[3mm]
&\hspace{2cm}\left.\left.\cdot\left(\frac{1}{W_0}\int_{-\pi}^{\alpha}e^{-\mathcal{H}\theta_A+\mu\mathcal{H}\theta_1+\omega_0\tilde{\tau}(\theta_A+\mu\theta_1)}\sin(\theta_A+\mu\theta_1+\omega_0
\tilde{\theta})\,d\alpha'{-1}\right)\right]\right]_{|\mu=0}\\[3mm]
&+qW_0\frac{\partial\theta_1}{\partial\alpha}-B\left[\frac{d}{d\mu}\left[e^{-\mathcal{H}\theta_A-\mu\mathcal{H}\theta_1+\omega_0\tilde{\tau}(\theta_A+\mu\theta_1)}\right]\right]_{|\mu=0}
\end{split}
\end{equation*}
\medskip

\noindent In order to compute $\frac{d}{d\mu}\tilde{\tau}$ we refer to \eqref{point-tilde-tau} and since it is multiply by $\omega_0$ that we will take equal to zero, then it will desapper as well as the terms multiplied by $p, B$. Thus the Fr\'echet derivative computed at $(\theta_c,\tilde{\omega}_c;0,0,0)$ is

\begin{equation}\label{D-theta-F1}
D_{\theta_A}\mathcal{F}_1(\theta_c,\tilde{\omega}_c;0,0,0)=\cosh(\mathcal{H}\theta_c)\mathcal{H}\theta_1+q\frac{\partial\theta_1}{\partial\alpha}. 
\end{equation}
\bigskip

\noindent The Fr\'echet derivative with respect to $\theta_A$, can be obtained by substituting the definition of the interface $z(\alpha)$ to the operator. Indeed, from the equations \eqref{x-y-Jacobian}, we get

\begin{equation}\label{interface-derivative}
\left\{\begin{array}{lll}
\displaystyle\frac{\partial z_1}{\partial\alpha}=-\frac{e^{-\tau(\alpha,1)}\cos(\theta(\alpha,1))}{W(\alpha,1)}\\[4mm]
\displaystyle\frac{\partial z_2}{\partial\alpha}=-\frac{e^{-\tau(\alpha,1)}\sin(\theta(\alpha,1))}{W(\alpha,1)},
\end{array}\right.
\end{equation}
\medskip

\noindent By substituting the value of $W(\alpha,1)$ for the point vortex and by rewriting $\tau, \theta$ as the sum of Crapper and a perturbation, then we have

\begin{equation}\label{point-interface}
\left\{\begin{array}{lll}
\displaystyle z_1(\alpha)=-\frac{1}{W_0}\int_{-\pi}^{\alpha} e^{-\mathcal{H}\theta_A-\omega_0\tilde{\tau}(\theta_A)}\cos(\theta_A+\omega_0
\tilde{\theta}(\theta_A))\,d\alpha'\\[4mm]
\displaystyle z_2(\alpha)=-\frac{1}{W_0}\int_{-\pi}^{\alpha} e^{-\mathcal{H}\theta_A-\omega_0\tilde{\tau}(\theta_A)}\sin(\theta_A+\omega_0
\tilde{\theta}(\theta_A))\,d\alpha'-1
\end{array}\right.
\end{equation}
\medskip

\noindent In a compact way, the interface $z(\alpha)$ is

\begin{equation}\label{point-z}
\displaystyle z(\alpha)=-\frac{1}{W_0}\int_{-\pi}^{\alpha} e^{-\mathcal{H}\theta_A-\omega_0\tilde{\tau}(\theta_A)+i(\theta_A+\omega_0\tilde{\theta}(\theta_A))}\,d\alpha' -e_2.
\end{equation}
\medskip

\noindent The main Fr\'echet derivative for the operator $\mathcal{F}_2$ is with respect to $\tilde{\omega}$.

\begin{equation*}
\begin{split}
&D_{\tilde{\omega}}\mathcal{F}_2=\left[\frac{d}{d\mu}\mathcal{F}_2(\theta_A,\tilde{\omega}+\mu\omega_1;B,p,\omega_0)\right]_{|\mu=0}=\left[\frac{d}{d\mu}\left[2 W_0 BR(z(\alpha),\tilde{\omega}(\alpha)+\mu\omega_1)\cdot\partial_{\alpha}z(\alpha)\right.\right.\\[3mm]
&\left.\left.+W_0 (\tilde{\omega}(\alpha)+\mu\omega_1(\alpha))+W_0\frac{\omega_0}{\pi}\frac{(z_2(\alpha),-z_1(\alpha))}{|z(\alpha)|^2}\cdot\partial_{\alpha}z(\alpha)+2\right]\right]_{|\mu=0}\\[3mm]
&=2W_0 \textrm{P.V.}\frac{1}{2\pi}\int_{-\pi}^{\pi}\frac{(z(\alpha)-z(\alpha'))^{\perp}}{|z(\alpha)-z(\alpha')|^2}\cdot \omega_1(\alpha')\,d\alpha'\cdot \partial_{\alpha}z(\alpha)+W_0\omega_1(\alpha).
\end{split}
\end{equation*}
\medskip

\noindent When we compute this derivative at the point $(\theta_c,\tilde{\omega}_c;0,0,0)$ we get

\begin{equation}\label{D-omega-F2}
D_{\tilde{\omega}}\mathcal{F}_2(\theta_c,\tilde{\omega}_c;0,0,0)=2BR(z^c(\alpha),\omega_1(\alpha))\cdot \partial_{\alpha}z^c(\alpha)+\omega_1(\alpha),
\end{equation}
\medskip

\noindent where $z^c(\alpha)$ is the parametrization of the Crapper interface coming from \eqref{Crapper-z-derivative}. 
\bigskip

\noindent The final step of this proof is to show the invertibility of the derivative's matrix, defined as follows

\begin{equation}\label{point-frechet-derivative}
D\mathcal{F}(\theta_c,\tilde{\omega}_c;0,0,0)=
\begin{pmatrix}
D_{\theta_A}\mathcal{F}_1 & 0  \\
D_{\theta_A}\mathcal{F}_2 &   D_{\tilde{\omega}}\mathcal{F}_2 
\end{pmatrix}=
\begin{pmatrix}
\Gamma & 0 \\
D_{\theta_A}\mathcal{F}_2 & \mathcal{A}(z^c(\alpha))+\mathcal{I}
\end{pmatrix}\cdot 
\begin{pmatrix}
\theta_1\\
\omega_1
\end{pmatrix}
\end{equation} 
\medskip

\noindent where 

\begin{align*}
&\Gamma\theta_1=\cosh(\mathcal{H}\theta_c)\mathcal{H}\theta_1+q\frac{d}{d\alpha}\theta_1\\[3mm]
&(\mathcal{A}(z^c(\alpha))+\mathcal{I})\omega_1=2BR(z^c(\alpha),\omega_1)\cdot\partial_{\alpha}z^c(\alpha)+\omega_1.
\end{align*}
\medskip

\noindent The invertibility of \eqref{point-frechet-derivative} is related with the invertibility of the diagonal, since the matrix is triangular. Hence we have to analyze the invertibility of the operators $\Gamma $ and $ \mathcal{A}+\mathcal{I}$, where $\mathcal{I}$ stays for the identity operator. Below, we resume the properties of the $\Gamma$ operator, for details, see in \cite{AAW2013} and \cite{CEG2016}. 

\begin{lemma}
The operator 
$$D_{\theta_A}\mathcal{F}_1(\theta_c, \tilde{\omega}_c;0,0,0)=\cosh(\mathcal{H}\theta_c)\theta_1+q\frac{d}{d\alpha}\theta_1=\Gamma\theta_1,$$
defined $\Gamma: H^{1}_{odd}\rightarrow L^{2}_{even}$ is injective.
\end{lemma}

\begin{proof}
The injectivity follows from the fact that $\Gamma\theta_1=0$ if and only if $\displaystyle\theta_1=\frac{d\theta_c}{d\alpha}$, see \cite[Lemma 2.1]{OS2001}. Moreover, we know that $\theta_c$ is an odd function then $\displaystyle\frac{d\theta_c}{d\alpha}$ is even. This statement implies that the constants are the only trivial solutions of $\Gamma\theta_1=0$.
\end{proof}

\noindent The problem concerning the invertibility of this operator is related with its surjectivity. 

\begin{lemma}\label{DF1-invertible}
Let $f\in L^{2}_{even}$. Then there exists $\theta_1\in H^{1}_{odd}$ with $\Gamma\theta_1=f$ if and only if 
$$(f,\cos\theta_c)=\int_{-\pi}^{\pi} f(\alpha)\cos\theta_c(\alpha)\,d\alpha=0$$
\end{lemma}

\begin{proof}
The complete proof can be found in \cite[Proposition 3.3]{AAW2013}. Here we will prove that the cokernel has dimension one and it is spanned by $\cos\theta_c$.\\
\noindent If we consider the operator $\mathcal{F}_1$ with  $(p,\omega_0,B)=(0,0,0)$, we have

\begin{align*}
&\int_{-\pi}^{\pi}\mathcal{F}_1\cos\theta\,d\alpha=\int_{-\pi}^{\pi}\left(\sin\mathcal{H}\theta + q\frac{d\theta}{d\alpha}\right)\cos\theta\,d\alpha=0,
\end{align*}

\noindent because the second term is the $q$ multiplied by $\sin\theta$ in the interval $(-\pi,\pi)$ and the the first term is $0$ because of the Cauchy integral theorem. In particular, if we take the derivative with respect to $\theta$ and we compute it in $\theta_c$ we get 

\begin{align*}
&\int_{-\pi}^{\pi}\Gamma\theta_1\cos\theta_c\,d\alpha +\int_{-\pi}^{\pi} \mathcal{F}_1\sin\theta_c\,d\alpha=\int_{-\pi}^{\pi}\Gamma\theta_1\cos\theta_c\,d\alpha+\int_{-\pi}^{\pi}\left(q\frac{d\theta_c}{d\alpha}+\sinh\mathcal{H}\theta_c\right)\sin\theta_c\,d\alpha=0,
\end{align*}

\noindent since the quantity in the brackets is $0$ for \eqref{Crapper-Bernoulli} thus it follows

\begin{equation}\label{Gamma-coseno}
\int_{-\pi}^{\pi}\Gamma\theta_1\cos\theta_c\,d\alpha=0.
\end{equation}
\end{proof}
\bigskip

\noindent For the operator $\mathcal{A}+\mathcal{I}$ we have the following result, proved in \cite{CCG2011}.

\begin{lemma}\label{DF2-invertible}
Let $z\in H^{3}$ be  a curve without self-intersections. Then

$$\mathcal{A}(z)\omega=2 BR(z,\omega)\cdot\partial_{\alpha}z$$
\medskip

\noindent defines a compact linear operator 

$$\mathcal{A}(z): H^{1}\rightarrow H^{1}$$
\medskip

\noindent whose eigenvalues are strictly smaller than $1$ in absolute value. In particular, the operator $\mathcal{A}+\mathcal{I}$ is invertible. 
\end{lemma}
\medskip

\noindent In conclusion, the equations

\begin{equation}
\begin{pmatrix}
\Gamma & 0 \\
D_{\theta_A}\mathcal{F}_2 & D_{\tilde{\omega}}\mathcal{F}_2 
\end{pmatrix}\cdot 
\begin{pmatrix}
\theta_1\\
\omega_1
\end{pmatrix}=
\begin{pmatrix}
f\\
g
\end{pmatrix}.
\end{equation} 
\medskip

\noindent computed at the point $(\theta_c,\tilde{\omega}_c;0,0,0)$  has a solution if and only if $|A|<A_0$ and $(f,\cos\theta_c)=0$.
\bigskip

\noindent To prove Theorem \ref{point-existence},  we cannot use directly the implicit function theorem \ref{IFT}  since the F\'echet derivative $D\mathcal{F}$ is not surjective. Following \cite{AAW2013} and also \cite{CEG2016}, we use an adaptation of the Lyapunov-Schmidt reduction argument. Define 

\begin{equation*}
\Pi\theta_1:=(\cos\theta_c,\theta_1)\frac{\cos\theta_c}{\|\cos\theta_c\|^2_{L^2}},
\end{equation*}
\medskip

\noindent where $\Pi$ is the $L^2$ projector onto the linear span of $\cos\theta_c$ and from \eqref{Gamma-coseno} we have $\Pi\Gamma=0$. Thus, we define the projector on $\Gamma(H^{2}_{odd})$, as $\mathcal {I}-\Pi$ and 

\begin{equation}\label{tilde-mathcal-F}
\tilde{\mathcal{F}}=((\mathcal {I}-\Pi)\mathcal{F}_1,\mathcal{F}_2):H^2_{odd}\times H^1_{even}\times\mathbb{R}^3\rightarrow \Gamma(H^2_{odd})\times L^2,
\end{equation}
\medskip

\noindent where $\mathcal{F}=(\mathcal{F}_1,\mathcal{F}_2)$ is defined in \eqref{point-water-waves}. The Fr\'echet derivatives of \eqref{tilde-mathcal-F} in $(\theta_A,\tilde{\omega}_A)$  at the Crapper point $(\theta_c,\tilde{\omega_c};0,0,0)$ is now invertible. So we can apply the implicit function theorem to $\tilde{\mathcal{F}}$ then there exists a smooth function $\Theta_c:U_{B,p,\omega_0}\rightarrow H^2_{odd}\times H^1_{even}$, where $U_{B,p,\omega_0}$ is a small neighborhod of $(0,0,0)$ such that $\Theta_c(0,0,0)=(\theta_c,\tilde{\omega}_c)$ and for all $(B,p,\omega_0)\in U_{B,p,\omega_0,}$

$$\tilde{\mathcal{F}}(\Theta_c(B,p,\omega_0);B,p,\omega_0)=0.$$
\medskip

\noindent But now, if we consider $\mathcal{F}(\Theta_c(B,p,\omega_0);B,p,\omega_0)$, defined in \eqref{point-water-waves}, then it could not be $0$. So we introduce a differentiable function on $U_{B,p,\omega_0}$:

$$f(B;p,\omega_0)=(\cos\theta_c,\mathcal{F}_1(\Theta_c(B,p,\omega_0);B,p,\omega_0)).$$
\medskip

\noindent We have that $\Pi\mathcal{F}_1=f(B;p,\omega_0)\frac{\cos\theta_c}{\|\cos\theta_c\|^2_{L^2}}$ and if we find a point $(B^*;p^*,\omega_0^*)$ such that $f(B^*;p^*,\omega_0^*)=0$, then 
$\mathcal{F}_1(\Theta_c(B^*,p^*,\omega_0^*);B^*,p^*,\omega_0^*))=0$ and so our problem is solved.\\

\noindent We note that choosing $(B^*;p^*,\omega_0^*)=(0,0,0)$, then  $f(0;0,0)=0$. Its derivative with respect to $B$ is 

\begin{equation*}
D_B f(0,0;0)=\left(\cos\theta_c,\Gamma\partial_k\Theta_c+e^{-\mathcal{H}\theta_c}\right)=\left(\cos\theta_c,e^{-\mathcal{H}\theta_c}\right)=-2\pi,
\end{equation*}
\medskip

\noindent where we have used \eqref{Gamma-coseno} and the Cauchy integral theorem. Hence, we can apply the implicit function theorem \ref{IFT} to the function $f$ and there exists a smooth function $B^*(p,\omega_0)$ that satisfies $f(B^*(p,\omega_0);p,\omega_0)=0$, for $(p,\omega_0)$ in $U_{p,\omega_0}$, a small neighborhod of $(0,0)$. \\

\noindent We can resume these results in the following theorem.

\begin{theorem}
Let $|A|<A_0$. There exist $(B,p,\omega_0)$  and a unique smooth function $B^*:U_{p,\omega_0}\rightarrow U_B$, such that  $B^*(0,0)=0$ and a unique smooth function 
$$\Theta_c: U_{B,p,\omega_0}\rightarrow H^2_{odd}\times H^1_{even},$$

such that $\Theta_c(0,0,0)=(\theta_c,\tilde{\omega}_c)$ and satisfy

$$\mathcal{F}(\Theta_c(B^*(p,\omega_0), p,\omega_0);B^*(p,\omega_0),p,\omega_0,)=0.$$
\end{theorem}
\medskip

\noindent The main Theorem \ref{point-existence} is a direct consequence of the theorem above. 

\bigskip

\section{The vortex patch case}\label{patch}
\subsection{Framework}
\noindent We consider a patch of vorticity $\omega(x,y)=\omega_0\chi_{D}(x,y)$, where $\omega_0\in\mathbb{R}$ and $\chi_{D}$ is the indicator function of the vortex domain $D$ near the origin, symmetric with respect to the $y$-axis, satisfying

\begin{equation}\label{small-distance-cond}
\max\{\textrm{dist}((x,y),(0,0))\}<<1, \quad\quad  \forall (x,y)\in \partial D,
\end{equation}
 
\noindent equivalent to consider a small vortex patch. In this case, as for the previous one, the fluid is incompressible then we introduce a stream function, which is the sum of an harmonic part $\psi_H$, defined in \eqref{harmonic-point-stream} and a part related to the vortex patch

\begin{equation}\label{patch-stream-function}
\begin{split}
\psi_{VP}(x,y)&=\frac{\omega_0}{2\pi}\int_{D}\log|(x,y)-(x',y')|\,dx'\,dy'.
\end{split}
\end{equation}
\medskip

\noindent We introduce the parametrization of the boundary $\partial D=\{\gamma(\alpha),\alpha\in [-\pi,\pi]\}$, that for now is a generic parametrization that satisfies the condition \eqref{small-distance-cond}. The representation of $D$ in the fig. \ref{new-annulus} is just an example, since it depends on the choice of the parametrization. Additionally, we obtain the velocity by taking the orthogonal gradient of the stream function, then the velocity associated to  \eqref{patch-stream-function} is

\begin{equation}\label{patch-velocity}
\begin{split}
v_{VP}(x,y)&=\frac{\omega_0}{2\pi}\int_{\partial D}\log|(x,y)-(\gamma_1(\alpha'),\gamma_2(\alpha'))|\partial_{\alpha}\gamma(\alpha')\,d\alpha'.
\end{split}
\end{equation}
\medskip

\noindent The velocity is the sum of \eqref{patch-velocity} and the orthogonal gradient of the harmonic stream function \eqref{harmonic-point-stream}

\begin{equation}\label{velocity-patch}
v(x,y)=(-\partial_y\psi_{H}(x,y), \partial_x\psi_H(x,y))+v_{VP}(x,y).
\end{equation}
\medskip

\noindent  As for the case of the point vortex we have to adapt the problem \eqref{v-euler}. One of the conditions is the kinematic boundary condition, so we need the velocity at the interface,

\begin{equation}
\begin{split}
v(z(\alpha))=&BR(z(\alpha),\tilde{\omega}(\alpha))+\frac{1}{2}\frac{\tilde{\omega}(\alpha)}{|\partial_{\alpha}z|^2}\partial_{\alpha}z+\frac{\omega_0}{2\pi}\int_{-\pi}^{\pi}\log|z(\alpha)-\gamma(\alpha')|\,\partial_{\alpha}\gamma(\alpha')\,d\alpha',
\end{split}
\end{equation}
\medskip

\noindent where $z(\alpha)$ is the parametrization of the interface $\partial\Omega$ and we can write the kinematic boundary condition as follows

 \begin{equation}\label{patch-kbc}
\begin{split}
v(z(\alpha))\cdot(\partial_{\alpha}z(\alpha))^{\perp}&=BR(z(\alpha),\tilde{\omega}(\alpha))\cdot\partial_{\alpha}z(\alpha)^{\perp}\\[3mm]
&+\frac{\omega_0}{2\pi}\int_{-\pi}^{\pi}\log|z(\alpha)-\gamma(\alpha')|\,\partial_{\alpha}\gamma(\alpha')\,d\alpha'\cdot (\partial_{\alpha}z(\alpha))^{\perp}=0.
\end{split}
\end{equation}
\medskip

\noindent In the analysis of this case we observe that the patch satisfies an elliptic equation and it has a moving boundary. We impose the patch to be fixed by the following condition

$$v(\gamma(\alpha))\cdot (\partial_{\alpha}\gamma(\alpha))^{\perp}=0.$$ 

\noindent This is equivalent to require

\begin{equation}\label{fix-patch}
\begin{split}
v(\gamma(\alpha))\cdot (\partial_{\alpha}\gamma(\alpha))^{\perp}&=\frac{1}{2\pi}\int_{-\pi}^{\pi} \frac{(\gamma(\alpha)-z(\alpha'))^{\perp}}{|\gamma(\alpha)-z(\alpha')|^2}\cdot\tilde{\omega}(\alpha')\,d\alpha'\cdot \partial_{\alpha}\gamma(\alpha)^{\perp}\\[3mm]
&+\frac{\omega_0}{2\pi} P.V. \int_{\partial D}\log|\gamma(\alpha)-\gamma(\alpha')|\partial_{\alpha}\gamma(\alpha')\,d\alpha'\cdot \partial_{\alpha}\gamma(\alpha)^{\perp}=0
\end{split} 
\end{equation}
 \medskip

\noindent Furthermore, we need another condition for identify completely our problem. This condition is related to the Bernoulli equation \eqref{v-euler5}, equivalent to \eqref{psi-eul3}. The most important issue is to fix the interface $\partial\Omega$, but for the vortex patch case we will slightly change the idea presented in subsection \ref{general-vorticity} and used for the point vortex case. To pass from the domain $\Omega(x,y)$ into $\tilde{\Omega}(\phi,\psi)$ we will consider an approximated stream function $\tilde{\psi}$ such that

\begin{equation}\label{approx-stream}
\left\{\begin{array}{lll}
\displaystyle\frac{\partial\tilde{\psi}}{\partial x}=W(x,y)\frac{\partial\psi}{\partial x}\\[2mm]
\displaystyle\frac{\partial\tilde{\psi}}{\partial y}=W(x,y)\frac{\partial\psi}{\partial y},
\end{array}\right.
\end{equation}
\medskip

\noindent hence, in this way $(\phi,\tilde{\psi})$ are in relation through the Cauchy-Riemann  equations

\begin{equation}\label{potential-approx-stream}
\left\{\begin{array}{lll}
\displaystyle\frac{\partial\phi}{\partial x}=\frac{\partial\tilde{\psi}}{\partial y}=W v_1\\[2mm]
\displaystyle\frac{\partial\phi}{\partial y}=-\frac{\partial\tilde{\psi}}{\partial x}=W v_2,
\end{array}\right.
\end{equation}
\medskip

\noindent where $W(x,y)$ has to satisfy \eqref{x-y-W} and it is equal to $1$ in the case of irrotational fluid and $\Delta\tilde{\psi}=0$. Moreover, we point out that the new domain $\tilde{\Omega}(\phi,\tilde{\psi})$ is also the lower half plane, see fig. \ref{new-annulus}, because the approximated stream function $\tilde{\psi}(z(\alpha))=0$, due to the kinematic boundary condition \eqref{v-euler4} and the positivity of $W$.\\

\noindent In view of the fact that we will use the new coordinate system $(\phi,\tilde{\psi})$, we have to rewrite \eqref{x-y-Jacobian} and \eqref{phi-psi-DBernoulli}. Let us start by writing the relation between $(x,y)$ and $(\phi,\tilde{\psi})$, so the system \eqref{x-y-Jacobian} becomes

\begin{equation}\label{x-y-phi-tilde_psi}
\begin{pmatrix}
\displaystyle\frac{\partial x}{\partial\phi} & \displaystyle\frac{\partial x}{\partial\tilde{\psi}}\\
\displaystyle\frac{\partial y}{\partial\phi} & \displaystyle\frac{\partial y}{\partial\tilde{\psi}}\\
\end{pmatrix}
=\frac{1}{W(v_1^2+v_2^2)} 
\begin{pmatrix}
v_1 & - v_2\\
v_2 &  v_1
\end{pmatrix},
\end{equation} 
\medskip

\noindent Now, we have to rewrite the Bernoulli equation \eqref{psi-eul3} in the new coordinates. We will derive \eqref{psi-eul3} with respect to $\phi$, such that the constant on the RHS will disappear. Thus we have

\begin{equation}\label{DBernoulli-patch2}
\frac{1}{2}\frac{\partial}{\partial\phi}\left(v_1^2+v_2^2\right)+p\frac{v_2}{W(v_1^2+v_2^2)}-q\frac{\partial}{\partial\phi}\left(\frac{W}{\sqrt{v_1^2+v_2^2}}\left(v_1\frac{\partial v_2}{\partial\phi}-v_2\frac{\partial v_1}{\partial\phi}\right)\right)=0.
\end{equation}
\medskip

\noindent Finally, it is natural to bring the equations into a disk, because of the periodicity of the problem. As we did for the point and we define

 \begin{equation}\label{patch-annulus}
\left\{\begin{array}{lll}
\phi=-\alpha\\[2mm]
\tilde{\psi}=\log\rho,
\end{array}\right.
\end{equation}
\medskip

\noindent And we pass from $\tilde{\Omega}(\phi,\tilde{\psi})$ into the unit disk, see fig. \ref{new-annulus}. We observe that the patch has been chosen symmetric with respect to the vertical axis, so in the coordinates $(\phi,\tilde{\psi})$ it remains symmetric with respect to the $\tilde{\psi}-$axis (due to the symmetry of the functions \eqref{symmetry}). And in the coordinates $(\alpha,\rho)$, by using \eqref{patch-annulus}, it will be symmetric with respect to the horizontal axis and contained in a circular sector, where $\pm\alpha_1$ are defined through $\pm\phi_1$, in this way

$$\textrm{dist}((\pm\phi_1,\tilde{\psi}_1),(0,\tilde{\psi}))>\textrm{dist}((\phi,\tilde{\psi}),(0,\tilde{\psi})), \quad \forall (\phi,\tilde{\psi})\in \partial{\tilde{D}}.$$
\medskip

 \begin{figure}[htbp]
\centering
\includegraphics[scale=0.5]{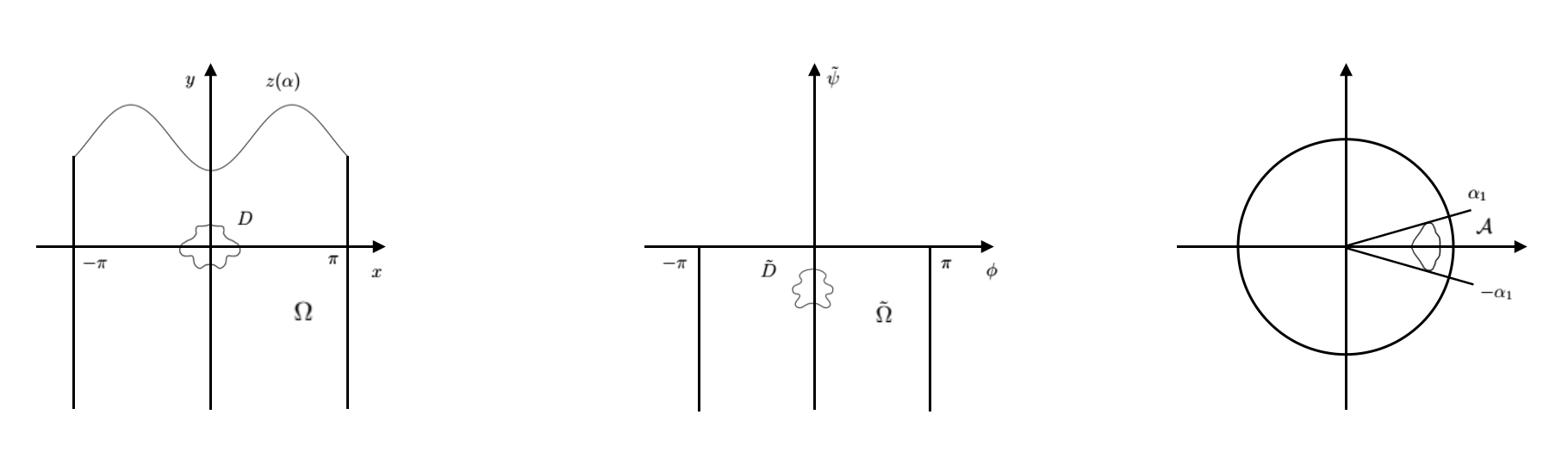}
\caption{The transformation of the patch $D(x,y)$, $\tilde{D}(\phi,\tilde{\psi})$ and $\mathcal{A}$.}\label{new-annulus}
\end{figure}
\medskip

\noindent In addition, by using the independent variables $(\tau,\theta)$, defined in  \eqref{tau-theta}, we write the equation $\Delta\tilde{\psi}=0$ in the new coordinates $(\phi,\tilde{\psi})$, by using the relations \eqref{approx-stream} and \eqref{potential-approx-stream} and we get an equation for $W(\phi,\tilde{\psi})$

$$\frac{\partial W}{\partial\tilde{\psi}}(v_1^2+v_2^2)-\omega_0\chi_{\tilde{D}}=0\quad\Rightarrow\quad W(\phi,\tilde{\psi})-W\left(\phi,-\infty\right) = \omega_0\int_{-\infty}^{\tilde{\psi}} e^{-2\tau(\phi,\tilde{\psi'})}\chi_{\tilde{D}}(\phi,\tilde{\psi})\,d\tilde{\psi'}. $$
\medskip

\noindent Since the value of $W$  at infinite is $1$, then in the variables $(\phi,\tilde{\psi})$, we have

\begin{equation}\label{W-phi-tilde-psi}
 W(\phi,\tilde{\psi})=1+\omega_0\int_{-\infty}^{\tilde{\psi}} e^{-2\tau(\phi,\tilde{\psi'})}\chi_{\tilde{D}}(\phi,\tilde{\psi'})\,d\tilde{\psi'}.
\end{equation}
\medskip

\noindent By using the change of variables \eqref{patch-annulus}, we have
 
\begin{equation}\label{patch-W-alpha-rho}
 W(\alpha,\rho)=1+\omega_0\int_{0}^{\rho}  \frac{e^{-2\tau(\alpha,\rho')}}{\rho'}\chi_{\mathcal{A}}(\alpha,\rho')\,d\rho'
 \end{equation}
\medskip

\noindent Concerning the derivative of the Bernoulli equation \eqref{DBernoulli-patch2}, we get

\begin{equation}\label{DBernoulli-patch3}
\frac{\partial}{\partial\alpha}\left( \frac{1}{2} e^{2\tau}\right)-p\frac{e^{-\tau}\sin\theta}{W}+q\frac{\partial}{\partial\alpha}\left(W e^{\tau}\frac{\partial\theta}{\partial\alpha}\right)=0.
\end{equation}
\medskip

\noindent By integrating with respect to $\alpha$ and taking the constant that appears on the RHS as $\frac{1}{2}+B$, as explained for the point vortex case, we get

\begin{equation*}
\frac{1}{2} e^{2\tau}-p\left(\int_{-\pi}^{\alpha}\frac{e^{-\tau}\sin\theta}{W(\alpha,1)}\,d\alpha'-1\right)+q W(\alpha,1) e^{\tau}\frac{\partial\theta}{\partial\alpha}=\frac{1}{2}+B,
\end{equation*}
\medskip

\noindent We multiply by $e^{-\tau}$ and we obtain the equation

\begin{equation}\label{patch-Bernoulli-tau-theta}
\sinh(\tau(\alpha,1))-p e^{-\tau(\alpha,1)}\left(\int_{-\pi}^{\alpha}\frac{e^{-\tau(\alpha',1)}\sin\theta(\alpha',1)}{W}\,d\alpha'-1\right)+q W \frac{\partial\theta(\alpha,1)}{\partial\alpha}-Be^{-\tau(\alpha,1)}=0.
\end{equation}
\medskip

\noindent We can solve our problem by finding a $2\pi$ periodic functions $\tau(\alpha)$ even and $\theta(\alpha)$ odd, an even function $\tilde{\omega}(\alpha)$ and a curve $\gamma(\alpha)$, which is the parametrization of the vortex patch that satisfy  \eqref{patch-kbc}, \eqref{fix-patch} and \eqref{patch-Bernoulli-tau-theta}. As we did for the point vortex, the kinematic boundary condition \eqref{patch-kbc}, can be replaced by

\begin{equation}\label{patch-v-tangential}
W(\alpha,1) v(z(\alpha))\cdot\partial_{\alpha}z(\alpha)=-1,
\end{equation}
\medskip

\noindent then our problem reduced to analyze the equation \eqref{fix-patch}, \eqref{patch-Bernoulli-tau-theta} and \eqref{patch-v-tangential}.
\bigskip

\subsection{Perturbation of the Crapper formulation with a vortex patch}
In this section, we want to write our variables as a perturbation of Crapper variables. First of all, we get a relation between $(\tau,\theta)$ in both $(\phi,\tilde{\psi})$ and $(\alpha,\rho)$ variables, by using the rotational and the divergence free conditions, 

\begin{equation}\label{tau-theta-patch}
\begin{cases}
\displaystyle\frac{\partial\theta}{\partial\phi}=\frac{\omega e^{-2\tau}}{W}+\frac{\partial\tau}{\partial\tilde{\psi}}\\[4mm]
\displaystyle\frac{\partial\theta}{\partial\tilde{\psi}}=-\frac{\partial\tau}{\partial\phi}
\end{cases}\Longrightarrow
\begin{cases}
\displaystyle -\frac{\partial\theta}{\partial\alpha}=\frac{\omega e^{-2\tau}}{W}+\rho\frac{\partial\tau}{\partial\rho}\\[4mm]
\displaystyle\rho\frac{\partial\theta}{\partial\rho}=\frac{\partial\tau}{\partial\alpha}.
\end{cases}
\end{equation}
\medskip

\noindent Once we find the values of $(\tau, \theta)$, we can use the relations \eqref{x-y-phi-tilde_psi} and \eqref{patch-annulus} to obtain the parametrization of the interface

\begin{equation}\label{patch-interface-derivative}
\left\{\begin{array}{lll}
\displaystyle\frac{\partial z_1}{\partial\alpha}=-\frac{e^{-\tau(\alpha,1)}\cos(\theta(\alpha,1))}{W(\alpha,1)}\\[4mm]
\displaystyle\frac{\partial z_2}{\partial\alpha}=-\frac{e^{-\tau(\alpha,1)}\sin(\theta(\alpha,1))}{W(\alpha,1)},
\end{array}\right.
\end{equation}
\medskip

\noindent where $W(\alpha,1)$ is defined in \eqref{patch-W-alpha-rho}.\\

\noindent In the case of rotational waves $(\tau,\theta)$ do not satisfy the Cauchy-Riemann equations. For this reason we define $\tau=\tau_A+\omega_0\tilde{\tau}$ and $\theta=\theta_A+\omega_0\tilde{\theta}$, such that $(\tau_A,\theta_A)$ is the Crapper solution with small gravity but without vorticity thus it is incompressible and irrotational and satisfies the Cauchy-Riemann equations in the variables $(\phi,\psi)$, as explained in \eqref{tauCR-thetaCR}
 
\begin{equation*}
\left\{\begin{array}{lll}
\displaystyle\frac{\partial\theta_A}{\partial\phi}=\frac{\partial\tau_A}{\partial{\psi}}\\[4mm]
\displaystyle\frac{\partial\theta_A}{\partial{\psi}}=-\frac{\partial\tau_A}{\partial\phi}.
\end{array}\right.
\end{equation*}
\medskip

\noindent This implies that on the interface $\mathcal{S}$, i.e. $\psi=0$, one variable can be written as the Hilbert transform of the other $\tau_A=\mathcal{H}\theta_A$. Hence, in the $(\phi,\tilde{\psi})$ variables, we have

\begin{equation}\label{Crapper-C-R-patch}
\begin{cases}
\displaystyle \frac{\partial\theta_A}{\partial\phi}=W\frac{\partial\mathcal{H}\theta_A}{\partial\tilde{\psi}}\\[4mm]
\displaystyle\frac{\partial\theta_A}{\partial\tilde{\psi}}=-\frac{1}{W
}\frac{\partial\mathcal{H}\theta_A}{\partial\phi}
\end{cases}\Longrightarrow
\begin{cases}
\displaystyle -\frac{\partial\theta_A}{\partial\alpha}=W\rho\frac{\partial\mathcal{H}\theta_A}{\partial\rho}\\[4mm]
\displaystyle\rho\frac{\partial\theta_A}{\partial\rho}=\frac{1}{W}\frac{\partial\mathcal{H}\theta_A}{\partial\alpha}.
\end{cases}
\end{equation}
\medskip

\noindent By substituting \eqref{Crapper-C-R-patch} in \eqref{tau-theta-patch}, we have 

\begin{equation}\label{patch-tilde-tau-theta-quasi-Cauchy-Riemann-1}
\left\{\begin{array}{lll}
\displaystyle\omega_0\frac{\partial\tilde{\theta}}{\partial\phi}=\left(\frac{1}{W}-1\right)\frac{\partial\theta_A}{\partial\phi}+\frac{\omega_0\chi_{\tilde{D}}e^{-2\tau}}{W}+\omega_0\frac{\partial\tilde{\tau}}{\partial\tilde{\psi}}\\[5mm]
\displaystyle \omega_0\frac{\partial\tilde{\theta}}{\partial\tilde{\psi}}=-\omega_0\frac{\partial\tilde{\tau}}{\partial\phi}+(W-1)\frac{\partial\theta_A}{\partial\tilde{\psi}}
\end{array}\right.
\end{equation}
\medskip

\begin{center}
$\Downarrow$
\end{center}
\medskip

\begin{equation}\label{patch-tilde-tau-theta-quasi-Cauchy-Riemann-2}
\left\{\begin{array}{lll}
\displaystyle-\omega_0\frac{\partial\tilde{\theta}}{\partial\alpha}=\left(1-\frac{1}{W}\right)\frac{\partial\theta_A}{\partial\alpha}+\frac{\omega_0\chi_{\tilde{D}}e^{-2\tau}}{W}+\rho\omega_0\frac{\partial\tilde{\tau}}{\partial\rho}\\[5mm]
\displaystyle \omega_0\frac{\partial\tilde{\theta}}{\partial\tilde{\rho}}=\frac{1}{\rho}\omega_0\frac{\partial\tilde{\tau}}{\partial\alpha}+\left(1+W\right)\frac{\partial\theta_A}{\partial\tilde{\rho}}
\end{array}\right.
\end{equation}
\medskip

\noindent By deriving with respect to the opposite variable and taking the difference, we have the following elliptic equation in $(\alpha,\rho)$

\begin{equation}\label{elliptic-alpha-rho-patch}
\begin{split}
\omega_0\frac{\partial^2\tilde{\tau}}{\partial\rho^2}+\frac{1}{\rho^2}\omega_0\frac{\partial^2\tilde{\tau}}{\partial\alpha^2}+\frac{1}{\rho}\omega_0\frac{\partial\tilde{\tau}}{\partial\rho}=&-\frac{1}{\rho}\frac{\partial}{\partial\rho}\left(\frac{\omega_0\chi_{\mathcal{A}}(\alpha,\rho)e^{-2\tau}}{W}\right)+\frac{1}{\rho}\frac{\partial}{\partial\rho}\left(\frac{1}{W}\right)\frac{\partial\theta_A}{\partial\alpha}\\[3mm]
&-\frac{1}{\rho}\frac{\partial W}{\partial\alpha}\frac{\partial\theta_A}{\partial\rho}+\frac{1}{\rho}\frac{\partial^2\theta_A}{\partial\alpha\partial\rho}\left(\frac{1-W^2}{W}\right).
\end{split}
\end{equation}
\medskip

\noindent But we are interested in the elliptic equation in $(\phi,\tilde{\psi})-$coordinates, since it will be easy to study and we have

\begin{equation}\label{elliptic-phi-tilde_psi-patch}
\begin{split}
\omega_0\Delta\tilde{\tau}=-\frac{\partial}{\partial\tilde{\psi}}\left(\frac{\omega_0\chi_{\tilde{D}}(\phi,\tilde{\psi}) e ^{-2\tau}}{W}\right)+\frac{1}{W^2}\frac{\partial W}{\partial\tilde{\psi}}\frac{\partial\theta_A}{\partial\phi}+\frac{\partial W}{\partial\phi}\frac{\partial\theta_A}{\partial\tilde{\psi}}+\left(\frac{W^2-1}{W}\right)\frac{\partial^2\theta_A}{\partial\phi\partial\tilde{\psi}}.
\end{split}
\end{equation}
\bigskip

\noindent We want to find a solution $\tilde{\tau}$ of the elliptic problem \eqref{elliptic-phi-tilde_psi-patch}. First of all let us rewrite the equation with all the explicit terms.

\begin{equation*}
\begin{split}
\Delta\tilde{\tau}(\phi,\tilde{\psi})=&-\frac{\partial}{\partial\tilde{\psi}}\left(\frac{\chi_{\tilde{D}}(\phi,\tilde{\psi}) e ^{(-2\mathcal{H}\theta_A-2\omega_0\tilde{\tau})(\phi,\tilde{\psi})}}{1+\omega_0\int_{-\infty}^{\tilde{\psi}}\chi_{\tilde{D}}(\phi,\tilde{\psi}')e^{(-2\mathcal{H}\theta_A-2\omega_0\tilde{\tau})(\phi,\tilde{\psi}')}\,d\tilde{\psi}'}\right)\\[4mm]
&+\frac{\chi_{\tilde{D}}(\phi,\tilde{\psi})e^{(-2\mathcal{H}\theta_A-2\omega_0\tilde{\tau})(\phi,\tilde{\psi})}}{\left(1+\omega_0\int_{-\infty}^{\tilde{\psi}}\chi_{\tilde{D}}(\phi,\tilde{\psi}')e^{(-2\mathcal{H}\theta_A-2\omega_0\tilde{\tau})(\phi,\tilde{\psi}')}\,d\tilde{\psi}'\right)^2}\frac{\partial\theta_A}{\partial\phi}\\[4mm]
&+\frac{\partial}{\partial\phi}\left(\int_{-\infty}^{\tilde{\psi}}\chi_{\tilde{D}}(\phi,\tilde{\psi}')e^{(-2\mathcal{H}\theta_A-2\omega_0\tilde{\tau})(\phi,\tilde{\psi}')}\,d\tilde{\psi}'\right)\frac{\partial\theta_A}{\partial\tilde{\psi}}\\[4mm]
&+\frac{2\int_{-\infty}^{\tilde{\psi}}\chi_{\tilde{D}}(\phi,\tilde{\psi}')e^{(-2\mathcal{H}\theta_A-2\omega_0\tilde{\tau})(\phi,\tilde{\psi}')}\,d\tilde{\psi}'}{1+\omega_0\int_{-\infty}^{\tilde{\psi}}\chi_{\tilde{D}}(\phi,\tilde{\psi}')e^{(-2\mathcal{H}\theta_A-2\omega_0\tilde{\tau})(\phi,\tilde{\psi}')}\,d\tilde{\psi}'}\frac{\partial^2\theta_A}{\partial\phi\partial\tilde{\psi}}\\[4mm]
&+\frac{\omega_0\left(\int_{-\infty}^{\tilde{\psi}}\chi_{\tilde{D}}(\phi,\tilde{\psi}')e^{(-2\mathcal{H}\theta_A-2\omega_0\tilde{\tau})(\phi,\tilde{\psi}')}\,d\tilde{\psi}'\right)^2}{1+\omega_0\int_{-\infty}^{\tilde{\psi}}\chi_{\tilde{D}}(\phi,\tilde{\psi}')e^{(-2\mathcal{H}\theta_A-2\omega_0\tilde{\tau})(\phi,\tilde{\psi}')}\,d\tilde{\psi}'}\frac{\partial^2\theta_A}{\partial\phi\partial\tilde{\psi}}=f((\phi,\tilde{\psi}),\tilde{\tau}).
\end{split}
\end{equation*}
\bigskip

\noindent where we use that $\displaystyle\frac{\partial W}{\partial\tilde{\psi}}=\omega_0\chi_{\tilde{D}}(\phi,\tilde{\psi})e^{(-2\mathcal{H}\theta_A-2\omega_0\tilde{\tau})(\phi,\tilde{\psi})}.$ Now, we define a solution in the following way

\begin{equation}\label{tau-tilde-solution}
\tilde{\tau}(\phi,\tilde{\psi})=f((\phi,\tilde{\psi}),\tilde{\tau})*G_2(\phi,\tilde{\psi}),
\end{equation}

\medskip

\noindent where $G_2(\phi,\tilde{\psi})$ is the Green function in the domain $\tilde{\Omega}$. We will show that \eqref{tau-tilde-solution} solves the elliptic equation, thanks to the smallness of the parameters involved.\\

\noindent If we use the properties of commutativity and differentiation of the convolution; the integration by parts with the fact that  $W(\pm\pi,\tilde{\psi}')=1$, then we are able to eliminate the derivative of $\tilde{\tau}$ and we have

\begin{equation}\label{patch-tilde-tau-omega0}
\begin{split}
\tilde{\tau}(\phi,\tilde{\psi})=&-\frac{\chi(\phi,\tilde{\psi}) e ^{(-2\mathcal{H}\theta_A-2\omega_0\tilde{\tau})(\phi,\tilde{\psi})}}{1+\omega_0\int_{-\infty}^{\tilde{\psi}}\chi_{\tilde{D}}(\phi,\tilde{\psi}')e^{(-2\mathcal{H}\theta_A-2\omega_0\tilde{\tau})(\phi,\tilde{\psi}')}\,d\tilde{\psi}'}*\frac{\partial}{\partial\tilde{\psi}} G_2(\phi,\tilde{\psi})\\[4mm]
&+\frac{\chi(\phi,\tilde{\psi}) e ^{(-2\mathcal{H}\theta_A-2\omega_0\tilde{\tau})(\phi,\tilde{\psi})}}{\left(1+\omega_0\int_{-\infty}^{\tilde{\psi}}\chi_{\tilde{D}}(\phi,\tilde{\psi}')e^{(-2\mathcal{H}\theta_A-2\omega_0\tilde{\tau})(\phi,\tilde{\psi}')}\,d\tilde{\psi}'\right)^2}\cdot\frac{\partial\theta_A}{\partial\phi}*G_2\\[4mm]
&-\frac{\partial G_2}{\partial\phi}*\frac{\partial\theta_A}{\partial\tilde{\psi}}\cdot\int_{-\infty}^{\tilde{\psi}}\chi_{\tilde{D}}(\phi,\tilde{\psi'})e^{(-2\mathcal{H}\theta_A-2\omega_0\tilde{\tau})(\phi,\tilde{\psi'})}\,d\tilde{\psi}'\\[4mm]
&-G_2*\frac{\partial^2\theta_A}{\partial\phi\partial\tilde{\psi}}\cdot\int_{-\infty}^{\tilde{\psi}}\chi_{\tilde{D}}(\phi,\tilde{\psi'})e^{(-2\mathcal{H}\theta_A-2\omega_0\tilde{\tau})(\phi,\tilde{\psi'})}\,d\tilde{\psi}'\\[4mm]
&+\frac{2\int_{-\infty}^{\tilde{\psi}}\chi_{\tilde{D}}(\phi,\tilde{\psi}')e^{(-2\mathcal{H}\theta_A-2\omega_0\tilde{\tau})(\phi,\tilde{\psi}')}\,d\tilde{\psi}'}{1+\omega_0\int_{-\infty}^{\tilde{\psi}}\chi_{\tilde{D}}(\phi,\tilde{\psi}')e^{(-2\mathcal{H}\theta_A-2\omega_0\tilde{\tau})(\phi,\tilde{\psi}')}\,d\tilde{\psi}'}\cdot \frac{\partial^2\theta_A}{\partial\phi\partial\tilde{\psi}}*G_2\\[4mm]
&+\frac{\omega_0\left(\int_{-\infty}^{\tilde{\psi}}\chi_{\tilde{D}}(\phi,\tilde{\psi}')e^{(-2\mathcal{H}\theta_A-2\omega_0\tilde{\tau})(\phi,\tilde{\psi}')}\,d\tilde{\psi}'\right)^2}{1+\omega_0\int_{-\infty}^{\tilde{\psi}}\chi_{\tilde{D}}(\phi,\tilde{\psi}')e^{(-2\mathcal{H}\theta_A-2\omega_0\tilde{\tau})(\phi,\tilde{\psi}')}\,d\tilde{\psi}'}\cdot \frac{\partial^2\theta_A}{\partial\phi\partial\tilde{\psi}}*G_2
\end{split}
\end{equation}
\medskip

\noindent Since we are looking for a solution with small $\omega_0$, we rewrite \eqref{patch-tilde-tau-omega0} around $\omega_0=0$, we write just the first order

\begin{equation}\label{approxim-tilde-tau}
\begin{split}
\tilde{\tau}&=-\chi(\phi,\tilde{\psi}) e ^{-2\mathcal{H}\theta_A(\phi,\tilde{\psi})}*\frac{\partial G_2}{\partial\tilde{\psi}}+\chi(\phi,\tilde{\psi}) e ^{-2\mathcal{H}\theta_A(\phi,\tilde{\psi})}\cdot\frac{\partial\theta_A}{\partial\phi}*G_2\\[4mm]
&-\frac{\partial G_2}{\partial\phi}*\frac{\partial\theta_A}{\partial\tilde{\psi}}\int_{-\infty}^{\tilde{\psi}}\chi(\phi,\tilde{\psi}') e ^{-2\mathcal{H}\theta_A(\phi,\tilde{\psi}')}\,d\tilde{\psi}'
\end{split}
\end{equation}

\begin{equation*}
\begin{split}
&-G_2*\frac{\partial^2\theta_A}{\partial\phi\partial\tilde{\psi}}\int_{-\infty}^{\tilde{\psi}}\chi(\phi,\tilde{\psi}') e ^{-2\mathcal{H}\theta_A(\phi,\tilde{\psi}')}\,d\tilde{\psi}'+2\int_{-\infty}^{\tilde{\psi}}\chi_{\tilde{D}}(\phi,\tilde{\psi}')e^{-2\mathcal{H}\theta_A(\phi,\tilde{\psi}')}\,d\tilde{\psi}'\cdot \frac{\partial^2\theta_A}{\partial\phi\partial\tilde{\psi}}*G_2\\[4mm]
&+\omega_02\chi_{\tilde{D}}(\phi,\tilde{\psi})e^{-2\mathcal{H}\theta_A(\phi,\tilde{\psi})}\tilde{\tau}(\phi,\tilde{\psi})*\frac{\partial G_2}{\partial\tilde{\psi}}\\[4mm]
&+\omega_0\chi_{\tilde{D}}(\phi,\tilde{\psi})e^{-2\mathcal{H}\theta_A(\phi,\tilde{\psi})}\int_{-\infty}^{\tilde{\psi}}\chi_{\tilde{D}}(\phi,\tilde{\psi}')e^{-2\mathcal{H}\theta_A(\phi,\tilde{\psi}')}\,d\tilde{\psi}'*\frac{\partial G_2}{\partial\tilde{\psi}}\\[4mm]
&-2\omega_0 \chi_{\tilde{D}}(\phi,\tilde{\psi})e^{-2\mathcal{H}\theta_A(\phi,\tilde{\psi})}\tilde{\tau}(\phi,\tilde{\psi})\cdot \frac{\partial\theta_A}{\partial\phi}*G_2\\[4mm]
&-2\omega_0\chi_{\tilde{D}}(\phi,\tilde{\psi})e^{-2\mathcal{H}\theta_A(\phi,\tilde{\psi})}\int_{-\infty}^{\tilde{\psi}}\chi_{\tilde{D}}(\phi,\tilde{\psi}')e^{-2\mathcal{H}\theta_A(\phi,\tilde{\psi}')}\,d\tilde{\psi}'\cdot \frac{\partial\theta_A}{\partial\phi}*G_2\\[4mm]
&+\omega_0\frac{\partial G_2}{\partial\phi}*\frac{\partial\theta_A}{\partial\tilde{\psi}}\cdot 2\int_{-\infty}^{\tilde{\psi}}\chi_{\tilde{D}}(\phi,\tilde{\psi}')e^{-2\mathcal{H}\theta_A(\phi,\tilde{\psi}')}\tilde{\tau}(\phi,\tilde{\psi}')\,d\tilde{\psi}'\\[4mm]
&+\omega_0 G_2*\frac{\partial^2\theta_A}{\partial\phi\partial\tilde{\psi}}\cdot 2 \int_{-\infty}^{\tilde{\psi}}\chi_{\tilde{D}}(\phi,\tilde{\psi}')e^{-2\mathcal{H}\theta_A(\phi,\tilde{\psi}')}\tilde{\tau}(\phi,\tilde{\psi}')\,d\tilde{\psi}'\\[4mm]
&-4\omega_0\int_{-\frac{\tilde{a}}{c}}^{\tilde{\psi}}\chi_{\tilde{D}}(\phi,\tilde{\psi}')e^{-2\mathcal{H}\theta_A(\phi,\tilde{\psi}')}\tilde{\tau}(\phi,\tilde{\psi}')\,d\tilde{\psi}'\cdot\frac{\partial^2\theta_A}{\partial\phi\partial\tilde{\psi}}*G_2\\[4mm]
&-\omega_0\left(\int_{-\infty}^{\tilde{\psi}}\chi_{\tilde{D}}(\phi,\tilde{\psi}')e^{-2\mathcal{H}\theta_A(\phi,\tilde{\psi}')}\,d\tilde{\psi}'\right)^2\cdot\frac{\partial^2\theta_A}{\partial\phi\partial\tilde{\psi}}*G_2+o(\omega_0^2)\\[4mm]
&\equiv\omega_0\mathcal{A}_1(\tilde{\tau},\theta_A)+\omega_0\mathcal{A}_2(\theta_A)+b(\theta_A)+o(\omega_0^2),
\end{split}
\end{equation*}
\medskip

\noindent We define the operator

\begin{equation}\label{G-operator}
\mathcal{G}(\tilde{\tau}; \omega_0,\theta_A)=\tilde{\tau}-
\omega_0\mathcal{A}_1(\tilde{\tau},\theta_A)-\omega_0\mathcal{A}_2(\theta_A)-b(\theta_A)+o(\omega_0^2),
\end{equation}
\medskip

\noindent where $\mathcal{G}(\tilde{\tau}; \omega_0,\theta_A):H^2_{even}\times\mathbb{R}\times H^2_{odd}\rightarrow H^2$ and to invert this operator in a neighborhod of $\omega_0=0$ we will use the Implicit function theorem \ref{IFT}. We observe that 

\begin{equation}\label{operator-tilde-tau}
\left\{\begin{array}{lll}
\mathcal{G}(\tilde{\tau};0,\theta_A) =0\\[4mm]
D_{\tilde{\tau}}\mathcal{G}(\tilde{\tau};0,\theta_A)=\tau_1.
\end{array}\right.
\end{equation}
\medskip

\noindent The equations \eqref{operator-tilde-tau} guarantees that in a neighborhod of $(\omega_0=0,\theta_A)$, there exists a smooth function $\tilde{\tau}^*(\omega_0,\theta_A)$, such that $\tilde{\tau}^*(0,\theta_A)=\tilde{\tau}$.

%
%
\bigskip

\subsection{Existence of Crapper waves in the presence of a small vortex patch}
In this section we prove the existence of a perturbation of the Crapper waves, with small gravity and small vorticity. We will prove the existence theorem (Theorem \ref{patch-existence}), by means of the implicit function theorem. However, to prove it we need an explicit parametrization for $\gamma(\alpha)$ in such a way that the operator, related to \eqref{fix-patch}, fulfils the hypotesis of the implicit function theorem. We define $\gamma(\alpha)$ as follows

\begin{equation}\label{gamma}
\gamma(\alpha)=
\left\{
\begin{array}{lll}
\displaystyle r\left(-\frac{\alpha+\pi}{\sin\alpha}\cos\alpha,-\alpha-\pi\right)\hspace{1.5cm}-\pi\leq \alpha<-\frac{\pi}{2}\\[5mm]
\displaystyle r\left(\frac{\alpha}{\sin\alpha}\cos\alpha,\alpha\right)\hspace{3.3cm}-\frac{\pi}{2}\leq \alpha\leq\frac{\pi}{2}\\[5mm]
\displaystyle r\left(-\frac{\alpha-\pi}{\sin\alpha}\cos\alpha,-\alpha+\pi\right)\hspace{2cm}\frac{\pi}{2}<\alpha\leq\pi,
\end{array}\right.
\end{equation}
\medskip

\noindent where $r\in \mathbb{R}$ the small radius. So that its derivative is

\begin{equation}
\partial_{\alpha}\gamma(\alpha)=
\left\{
\begin{array}{lll}
\displaystyle r\left(\frac{(\alpha+\pi)-\cos\alpha\sin\alpha}{\sin^2\alpha},-1\right)\hspace{1.5cm}-\pi\leq \alpha<-\frac{\pi}{2}\\[5mm]
\displaystyle r\left(\frac{\cos\alpha\sin\alpha-\alpha}{\sin^2\alpha},1\right)\hspace{3.3cm}-\frac{\pi}{2}\leq \alpha\leq\frac{\pi}{2}\\[5mm]
\displaystyle r\left(\frac{(\alpha-\pi)-\cos\alpha\sin\alpha}{\sin^2\alpha},-1\right)\hspace{2cm}\frac{\pi}{2}<\alpha\leq\pi
\end{array}\right.
\end{equation}
\medskip
 
 \begin{figure}[htbp]
\centering
\includegraphics[scale=0.8]{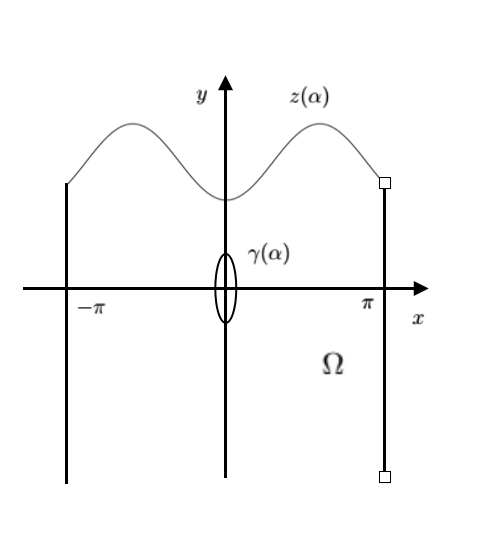}
\caption{The choice of $\gamma(\alpha)$.}\label{my-patch}
\end{figure}

\medskip
 
\noindent However,  we will work in a neighborhood of $\omega_0=0$, then we substitute $\tau=\mathcal{H}\theta_A+\omega_0 \tilde{\tau}^*$. 

\begin{subequations}
\label{patch-water-waves}
\begin{align}
\begin{split}
&\mathcal{F}_1(\theta_A, \tilde{\omega}, r; B,p,\omega_0)=\sinh(\mathcal{H}\theta_A+\omega_0\tilde{\tau}^*)\\[3mm]
&\hspace{2cm}-p e^{-\mathcal{H}\theta_A-\omega_0\tilde{\tau}^*}\left(\int_{-\pi}^{\alpha}\frac{e^{-\mathcal{H}\theta_A-\omega_0\tilde{\tau}^*}\sin\theta(\alpha',1)}{W(\alpha',1)}\,d\alpha'-1\right)\\[3mm]
&\hspace{2cm}+q W(\alpha,1) \frac{\partial(\theta_A+\omega_0\tilde{\theta}^*)}{\partial\alpha}-Be^{-\mathcal{H}\theta_A-\omega_0\tilde{\tau}^*}.
\end{split}\label{patch-water-waves-1}\\[5mm]
\begin{split}
&\mathcal{F}_2(\theta_A,\tilde{\omega}, r; B,p,\omega_0):=W(\alpha,1)\left(2BR(z(\alpha),\tilde{\omega}(\alpha))\cdot\partial_{\alpha}z(\alpha)+\tilde{\omega}(\alpha)\right.\\[3mm]
&\hspace{2cm}+\left.\frac{\omega_0}{2\pi}\int_{-\pi}^{\pi}\log|z(\alpha)-\gamma(\alpha')|\partial_{\alpha}\gamma(\alpha')\,d\alpha'\cdot\partial_{\alpha}z(\alpha)\right)+2
\end{split}\label{patch-water-waves-2}\\[5mm]
\begin{split}
&\mathcal{F}_3(\theta_A,\tilde{\omega}, r; B, p,\omega_0):=\frac{1}{2\pi}\int_{-\pi}^{\pi}\frac{(\gamma(\alpha)-z(\alpha'))^{\perp}}{|\gamma(\alpha)-z(\alpha')|^2}\cdot\tilde{\omega}(\alpha')\,d\alpha'\cdot \partial_{\alpha}\gamma(\alpha)^{\perp}\\[3mm]
&\hspace{2cm}+\frac{\omega_0}{2\pi}P.V.\int_{-\pi}^{\pi}\log|\gamma(\alpha)-\gamma(\alpha')|\partial_{\alpha}\gamma(\alpha')\,d\alpha'\cdot\partial_{\alpha}\gamma(\alpha)^{\perp}
\end{split}\label{patch-water-waves-3}
\end{align}
\end{subequations}
\bigskip

\noindent We have that  
$$(\mathcal{F}_1,\mathcal{F}_2,\mathcal{F}_3)(\theta_A, \tilde{\omega}, r; B,p,\omega_0):H^{2}_{odd}\times H^{1}_{even}\times\mathbb{R}^{{4}}\rightarrow H^{1}_{even}\times H^{1}_{even}\times H^{1}.$$
\medskip

\noindent The main theorem we want to prove is the following

\begin{theorem}\label{patch-existence}
Let us consider the water waves problem \eqref{v-euler}, with a small vortex patch and a small gravity $g$. Then, for some values of $A<A_0$, defined in  \eqref{Crapper-solutions}, there exist periodic solutions to \eqref{v-euler} with overhanging profile.
\end{theorem}

\bigskip

\subsubsection{Proof of Theorem \ref{patch-existence}}
We will analyse the three operators \eqref{patch-water-waves} that identify our problem. And we will show they satisfy the hypotesis of the implicit function theorem. First of all we have to show that 

\begin{equation}\label{operators-zero}
(\mathcal{F}_1,\mathcal{F}_2,\mathcal{F}_3)(\theta_c,\tilde{\omega}_c,0;0,0,0)=(0,0,0).
\end{equation}
\medskip

\noindent For $\mathcal{F}_1$, we use \eqref{Crapper-Bernoulli}

\begin{equation*}
\begin{split}
\mathcal{F}_1(\theta_c,\tilde{\omega}_c,0;0,0,0)&=\sinh(\mathcal{H}\theta_c)+q\frac{\partial\theta_c}{\partial\alpha}=0
\end{split}
\end{equation*}
\medskip

\noindent For $\mathcal{F}_2$ it holds by construction \eqref{patch-v-tangential}. For $\mathcal{F}_3$, we write explicitly $\gamma(\alpha)$ as in \eqref{gamma} and by taking the radius $r$ to be $0$. Thus $\mathcal{F}_3(\theta_c,\tilde{\omega}_c,0;0,0,0,0)$ satisfies \eqref{operators-zero}.\\

\noindent The most considerable part is to prove the invertibility of the derivatives. We observe that $D_{\tilde{\omega}}\mathcal{F}_1=D_{r}\mathcal{F}_1=0$, so it remains to compute $D_{\theta_A}\mathcal{F}_1$.

\begin{equation*}
\begin{split}
&D_{\theta_A}\mathcal{F}_1=\frac{d}{d\mu}\left[\sinh(\mathcal{H}\theta_A+\mu\mathcal{H}\theta_1+\omega_0\tilde{\tau}^*(\theta_A+\mu\theta_1))+qW_{\mu}\frac{\partial(\theta_A+\mu\theta_1+\omega_0\tilde{\theta}^*(\theta_A+\mu\theta_1))}{\partial\alpha}\right.\\[3mm]
&\left.-pe^{-\mathcal{H}\theta_A-\mu\mathcal{H}\theta_1+\omega_0\tilde{\tau}^*(\theta_A+\mu\theta_1)}\int_{-\pi}^{\alpha}\frac{e^{-\mathcal{H}\theta_A-\mu\mathcal{H}\theta_1-\omega_0\tilde{\tau}^*(\theta_A+\mu\theta_1)}\sin(\theta_A+\mu\theta_1+\omega_0\tilde{\theta}^*(\theta_A+\mu\theta_1))}{W_\mu}
\,d\alpha'\right.\\[3mm]
&\left.-Be^{-\mathcal{H}\theta_A-\mu\mathcal{H}\theta_1-\omega_0\tilde{\tau}^*(\theta_A+\mu\theta_1)}\right]_{|\mu=0}\\[5mm]
&=\cosh(\mathcal{H}\theta_A+\omega_0\tilde{\tau}^*(\theta_A))\cdot\left(\mathcal{H}\theta_1+\omega_0\left[\frac{d}{d\mu}\tilde{\tau}^*(\theta_A+\mu\theta_1)\right]_{|\mu=0}\right)\\[3mm]
&+ q\left[\frac{d}{d\mu}W_{\mu}\right]_{|\mu=0}\frac{\partial(\theta_A+\omega_0\tilde{\theta}^*(\theta_A))}{\partial\alpha}+q W\frac{\partial}{\partial\alpha}\left(\theta_1+\omega_0\left[\frac{d}{d\mu}\tilde{\theta}^*(\theta_A+\mu\theta_1)\right]_{|\mu=0}\right)\\[3mm]
&+p e^{-\mathcal{H}\theta_A-\omega_0\tilde{\tau}^*(\theta_A)}\left(\mathcal{H}\theta_1+\omega_0\left[\frac{d}{d\mu}\tilde{\tau}^*(\theta_A+\mu\theta_1)\right]_{|\mu=0}\right)\cdot\int_{-\pi}^{\alpha}\frac{e^{-\mathcal{H}\theta_A-\omega_0\tilde{\tau}^*(\theta_A)}\sin(\theta_A+\omega_0\tilde{\theta}^*(\theta_A))}{W}
\,d\alpha'\\[5mm]
&-p e^{-\mathcal{H}\theta_A-\omega_0\tilde{\tau}^*(\theta_A)}\int_{-\pi}^{\alpha}\frac{e^{-\mathcal{H}\theta_A-\omega_0\tilde{\tau}^*(\theta_A)}}{W}\left(-\mathcal{H}\theta_1-\omega_0\left[\frac{d}{d\mu}\tilde{\tau}^*(\theta_A+\mu\theta_1)\right]_{|\mu=0}\right)\sin(\theta_A+\omega_0\tilde{\theta}^*(\theta_A))\,d\alpha'\\[5mm]
&-p e^{-\mathcal{H}\theta_A-\omega_0\tilde{\tau}^*(\theta_A)}\int_{-\pi}^{\alpha}\frac{e^{-\mathcal{H}\theta_A-\omega_0\tilde{\tau}^*(\theta_A)}}{W}\cos(\theta_A+\omega_0\tilde{\theta}^*(\theta_A))\left(\theta_1+\omega_0\left[\frac{d}{d\mu}\tilde{\theta}^*(\theta_A+\mu\theta_1)\right]_{|\mu=0}\right)\,d\alpha'\\[5mm]
&+p e^{-\mathcal{H}\theta_A-\omega_0\tilde{\tau}^*(\theta_A)}\int_{-\pi}^{\alpha}\frac{e^{-\mathcal{H}\theta_A-\omega_0\tilde{\tau}^*(\theta_A)}\sin(\theta_A+\omega_0\tilde{\theta}^*(\theta_A))}{W^2}\left[\frac{d}{d\mu}W_{\mu}\right]_{|\mu=0}\,d\alpha'\\[3mm]
&-Be^{-\mathcal{H}\theta_A-\omega_0\tilde{\tau}^*(\theta_A)}\left(-\mathcal{H}\theta_1-\omega_0\left[\frac{d}{d\mu}\tilde{\tau}^*(\theta_A+\mu\theta_1)\right]_{|\mu=0}\right).
\end{split}
\end{equation*}
\medskip

\begin{remark}\label{tilde-tau-frechet-derivative}
\noindent The equation for $W_{\mu}(\alpha,1)$ is the following

\begin{equation}\label{patch-W-mu}
W_{\mu}(\alpha,1)=1+\int_{0}^{1} \omega_0\chi_{\tilde{D}_{\mathcal{A}}}(\alpha,\rho') \frac{e^{-2\mathcal{H}\theta_A-2\mu\mathcal{H}\theta_1-2\tilde\omega_0\tilde{\tau}^*(\theta_A+\mu\theta_1)}}{\rho'}\,d\rho'
\end{equation}
\bigskip

\noindent Now we have to compute $\displaystyle\frac{d W_{\mu}}{d\mu}$, that is

\begin{equation}\label{patch-DW-mu}
\begin{split}
&\left[\frac{d W_{\mu}}{d\mu}\right]_{|\mu=0}=\left[\frac{d}{d\mu}\left(1+\int_{0}^{1} \omega_0\chi_{\tilde{D}_{\mathcal{A}}}(\alpha,\rho') \frac{e^{-2\mathcal{H}\theta_A-2\mu\mathcal{H}\theta_1-2\omega_0\tilde{\tau}^*(\theta_A+\mu\theta_1)}}{\rho'}\,d\rho'\right)\right]_{|\mu=0}\\[5mm]
&=\left[\int_{0}^{1}\omega_0\chi_{\tilde{D}_{\mathcal{A}}}(\alpha,\rho') \frac{e^{-2\mathcal{H}\theta_A-2\mu\mathcal{H}\theta_1-2\omega_0\tilde{\tau}^*(\theta_A+\mu\theta_1)}}{\rho'}\cdot\right.\\[4mm]
&\hspace{3cm}\left.\cdot\left(-2\mathcal{H}\theta_1-2\omega_0\frac{d}{d\mu}\tilde{\tau}^*(\theta_A+\mu\theta_1)\right)\,d\rho'\right]_{|\mu=0}\\[5mm]
&=\omega_0\int_{0}^{1}\chi_{\tilde{D}_{\mathcal{A}}}(\alpha,\rho') \frac{e^{-2\mathcal{H}\theta_A-2\omega_0\tilde{\tau}^*(\theta_A)}}{\rho'}\left(-2\mathcal{H}\theta_1-2\omega_0\left[\frac{d}{d\mu}\tilde{\tau}^*(\theta_A+\mu\theta_1)\right]_{\mu=0}\right)\,d\rho'
\end{split}
\end{equation}
\medskip

\noindent It remains to observe that for our purpose it is sufficient to have the existence of $\frac{d}{d\mu}\tilde{\tau}^*(\theta_A+\mu\theta_1)$, coming from the elliptic equation \eqref{elliptic-alpha-rho-patch}. Indeed we must compute the Fr\'echet derivative at the point $(\theta_c,\tilde{\omega_c},0;0,0,0)$ and, as we can see in $D_{\theta_A}\mathcal{F}_1$, the term $\displaystyle\left[\frac{d}{d\mu}\tilde{\tau}(\theta_A+\mu\theta_1)\right]_{|\mu=0}$ is always multiplied by $\omega_0$ that it is taken equal to zero. And we can state that also $\displaystyle \left[\frac{d W_{\mu}}{d\mu}\right]_{|\mu=0}$ is zero.
\end{remark}
\medskip

\noindent The remark \ref{tilde-tau-frechet-derivative} implies that

\begin{equation*}
D_{\theta_A}\mathcal{F}_1(\theta_c,\tilde{\omega}_c,0;0,0,0)=
\cosh(\mathcal{H}\theta_c)\cdot\mathcal{H}\theta_1+q\frac{\partial\theta_1}{\partial\alpha}.
\end{equation*}
\medskip

\noindent For the second operator, we observe that for computing $D_{\theta_A}\mathcal{F}_2$, we need to use the equation for $z(\alpha)$ can be obtained by integrating \eqref{patch-interface-derivative} and we define 

$$z_\mu(\alpha)=-\int_{-\pi}^{\alpha} \frac{e^{-\mathcal{H}\theta_A-\mu\mathcal{H}\theta_1-\omega_0\tilde{\tau}^*(\theta_A+\mu\theta_1)+i(\theta_A+\mu\theta_1+\omega_0
\tilde{\theta}^*(\theta_A+\mu\theta_1))}}{W_{\mu}(\alpha',1)}\,d\alpha'-e_2$$
\medskip

\noindent where $W_{\mu}(\alpha',1)$ is defined in \eqref{patch-W-mu}. \\

\noindent In the same way, we did for computing $D_{\theta_A}\mathcal{F}_1$, we can compute $D_{\theta_A}\mathcal{F}_2$ and then at the point $(\theta_c,\tilde{\omega}_c,0;0,0,0)$ we will get $D_{\theta_A}\mathcal{F}_2(\theta_c,\tilde{\omega}_c,0;0,0,0)$.\\

\noindent Instead, it is important to compute $D_{\tilde{\omega}}\mathcal{F}_2$.

\begin{equation*}
\begin{split}
&D_{\tilde{\omega}}\mathcal{F}_2=\left[\frac{d}{d\mu}\mathcal{F}_2(\theta_A,\tilde{\omega}+\mu\omega_1,r;p,\varepsilon,\omega_0,B)\right]_{|\mu=0}=\left[\frac{d}{d\mu}\left[2 W(\alpha,1) BR(z(\alpha),\tilde{\omega}(\alpha)+\mu\omega_1)\cdot\partial_{\alpha}z(\alpha)\right.\right.\\[3mm]
&\left.\left.+W(\alpha,1) (\tilde{\omega}(\alpha)+\mu\omega_1(\alpha))+W(\alpha,1)\frac{\omega_0}{2\pi}\int_{-\pi}^{\pi}\log|z(\alpha)-\gamma(\alpha')|\partial_{\alpha}\gamma(\alpha')\,d\alpha'\cdot\partial_{\alpha}z(\alpha)+2\right]\right]_{|\mu=0}\\[3mm]
&=2W(\alpha,1)\textrm{P.V.}\frac{1}{2\pi}\int_{-\pi}^{\pi}\frac{(z(\alpha)-z(\alpha'))^{\perp}}{|z(\alpha)-z(\alpha')|^2}\cdot \omega_1(\alpha')\,d\alpha'\cdot \partial_{\alpha}z(\alpha)+W(\alpha,1)\omega_1(\alpha).
\end{split}
\end{equation*}
\medskip

\noindent At the Crapper point we have

$$D_{\tilde{\omega}}\mathcal{F}_2(\theta_c,\tilde{\omega}_c,0;0,0,0,0)=2\textrm{P.V.}\frac{1}{2\pi}\int_{-\pi}^{\pi}\frac{(z(\alpha)-z(\alpha'))^{\perp}}{|z(\alpha)-z(\alpha')|^2}\cdot \omega_1(\alpha')\,d\alpha'\cdot \partial_{\alpha}z(\alpha)+\omega_1(\alpha).$$
\bigskip

\noindent It remains to compute the last derivate 
\begin{equation*}
\begin{split}
&D_{r}\mathcal{F}_2=\frac{d}{d\mu}\left[\mathcal{F}_2(\theta_A,\tilde{\omega},r+\mu r_1;B,p,\omega_0)\right]_{|\mu=0}\\[3mm]
&=\left[\frac{d}{d\mu}\left[2 W(\alpha,1) BR(z(\alpha),\tilde{\omega}(\alpha))\cdot\partial_{\alpha}z(\alpha)+W(\alpha,1)\tilde{\omega}(\alpha)\right.\right.\\[3mm]
&+W(\alpha,1) \frac{\omega_0}{2\pi}\int_{-\pi}^{-\frac{\pi}{2}}\log\sqrt{\left(z_1(\alpha)+(r+\mu r_1)\frac{\alpha'+\pi}{\sin\alpha'}\cos\alpha'\right)^2+\left(z_2(\alpha)+(r+\mu r_1)(\alpha'+\pi)\right)^2}\\[3mm]
&\hspace{3cm}\cdot (r+\mu r_1)\left(\frac{(\alpha'+\pi)-\cos\alpha'\sin\alpha'}{\sin^2\alpha'},-1\right)\,d\alpha' \cdot\partial_{\alpha}z(\alpha)\\[3mm]
&+W(\alpha,1)\frac{\omega_0}{2\pi}\int_{-\frac{\pi}{2}}^{\frac{\pi}{2}}\log\sqrt{\left(z_1(\alpha)-(r+\mu r_1)\frac{\alpha'}{\sin\alpha'}\cos\alpha'\right)^2+\left(z_2(\alpha)-(r+\mu r_1)(\alpha')\right)^2}\\[3mm]
&\hspace{3cm}\cdot (r+\mu r_1)\left(\frac{\cos\alpha'\sin\alpha'-\alpha'}{\sin^2\alpha'},1\right)\,d\alpha' \cdot\partial_{\alpha}z(\alpha)\\[4mm]
&+W(\alpha,1)\frac{\omega_0}{2\pi}\int_{\frac{\pi}{2}}^{\pi}\log\sqrt{\left(z_1(\alpha)+(r+\mu r_1)\frac{\alpha'-\pi}{\sin\alpha'}\cos\alpha'\right)^2+\left(z_2(\alpha)+(r+\mu r_1)(\alpha'-\pi)\right)^2}\\[3mm]
&\hspace{3cm}\left.\cdot (r+\mu r_1)\left(\frac{(\alpha'-\pi)-\cos\alpha'\sin\alpha'}{\sin^2\alpha'},-1\right)\,d\alpha' \cdot\partial_{\alpha}z(\alpha)\right]_{|\mu=0}\\[4mm]
\end{split}
\end{equation*}
\begin{equation*}
\begin{split}
&=W(\alpha,1)\frac{\omega_0}{2\pi}r\cdot r_1\int_{-\pi}^{-\frac{\pi}{2}}\frac{z_1(\alpha)+r\frac{\alpha'+\pi}{\sin\alpha'}\cos\alpha'+z_2(\alpha)+r(\alpha'+\pi)}{\left(z_1(\alpha)+r\frac{\alpha'+\pi}{\sin\alpha'}\cos\alpha'\right)^2+\left(z_2(\alpha)+r(\alpha'+\pi)\right)^2}\\[3mm]
&\hspace{3cm}\cdot\left(\frac{(\alpha'+\pi)-\cos\alpha'\sin\alpha'}{\sin^2\alpha'},-1\right)\,d\alpha' \cdot\partial_{\alpha}z(\alpha)\\[3mm]
&+W(\alpha,1)\frac{\omega_0}{2\pi}\cdot r_1\int_{-\pi}^{-\frac{\pi}{2}}\log\sqrt{\left(z_1(\alpha)+r\frac{\alpha'+\pi}{\sin\alpha'}\cos\alpha'\right)^2+\left(z_2(\alpha)+r(\alpha'+\pi)\right)^2}\\[3mm]
&\hspace{3cm}\cdot\left(\frac{(\alpha'+\pi)-\cos\alpha'\sin\alpha'}{\sin^2\alpha'},-1\right)\,d\alpha' \cdot\partial_{\alpha}z(\alpha)\\[3mm]
&+W(\alpha,1)\frac{\omega_0}{2\pi}r\cdot r_1\int_{-\frac{\pi}{2}}^{\frac{\pi}{2}}\frac{z_1(\alpha)-r\frac{\alpha'}{\sin\alpha'}\cos\alpha'+z_2(\alpha)-r(\alpha')}{\left(z_1(\alpha)-r\frac{\alpha'}{\sin\alpha'}\cos\alpha'\right)^2+\left(z_2(\alpha)-r\alpha'\right)^2}\\[3mm]
&\hspace{3cm}\cdot\left(\frac{\cos\alpha'\sin\alpha'-\alpha'}{\sin^2\alpha'},1\right)\,d\alpha' \cdot\partial_{\alpha}z(\alpha)\\[3mm]
&+W(\alpha,1)\frac{\omega_0}{2\pi}\cdot r_1\int_{-\frac{\pi}{2}}^{\frac{\pi}{2}}\log\sqrt{\left(z_1(\alpha)-r\frac{\alpha'}{\sin\alpha'}\cos\alpha'\right)^2+\left(z_2(\alpha)-r\alpha'\right)^2}\\[3mm]
&\hspace{3cm}\cdot\left(\frac{\cos\alpha'\sin\alpha'-\alpha'}{\sin^2\alpha'},1\right)\,d\alpha' \cdot\partial_{\alpha}z(\alpha)
\end{split}
\end{equation*}
\begin{equation*}
\begin{split}
&+W(\alpha,1)\frac{\omega_0}{2\pi}r\cdot r_1\int_{\frac{\pi}{2}}^{\pi}\frac{z_1(\alpha)+r\frac{\alpha'-\pi}{\sin\alpha'}\cos\alpha'+z_2(\alpha)+r(\alpha'-\pi)}{\left(z_1(\alpha)+r\frac{\alpha'-\pi}{\sin\alpha'}\cos\alpha'\right)^2+\left(z_2(\alpha)+r(\alpha'-\pi)\right)^2}\\[3mm]
&\hspace{3cm}\cdot\left(\frac{(\alpha'-\pi)-\cos\alpha'\sin\alpha'}{\sin^2\alpha'},-1\right)\,d\alpha' \cdot\partial_{\alpha}z(\alpha)\\[3mm]
&+W(\alpha,1)\frac{\omega_0}{2\pi}\cdot r_1\int_{\frac{\pi}{2}}^{\pi}\log\sqrt{\left(z_1(\alpha)+r\frac{\alpha'-\pi}{\sin\alpha'}\cos\alpha'\right)^2+\left(z_2(\alpha)+r(\alpha'-\pi)\right)^2}\\[3mm]
&\hspace{3cm}\cdot\left(\frac{(\alpha'-\pi)-\cos\alpha'\sin\alpha'}{\sin^2\alpha'},-1\right)\,d\alpha' \cdot\partial_{\alpha}z(\alpha)
\end{split}
\end{equation*}
\medskip

\noindent When we evaluate this derivative at $(\theta_c,\tilde{\omega}_c,0;0,0,0)$, we get

$$D_{r}\mathcal{F}_2(\theta_c,\tilde{\omega}_c,0;0,0,0)=0.$$
\medskip

\noindent For the last operator $\mathcal{F}_3$ we have to compute the derivates, but for $D_{\theta_A}\mathcal{F}_3$ and $D_{\tilde{\omega}}\mathcal{F}_3$, we have just to substitute $\theta_A\mapsto \theta_A+\mu\theta_1$ and $\tilde{\omega}\mapsto\tilde{\omega}+\mu\omega_1$, respectively and compute the derivatives as we did for the previous operators. Then we will compute them at the Crapper point, so that we get $D_{\theta_A}\mathcal{F}_3(\theta_c,\tilde{\omega}_c,0;0,0,0)$ and $D_{\tilde{\omega}}\mathcal{F}_3(\theta_c,\tilde{\omega}_c,0;0,0,0)$. In order to apply the implicit function theorem the relevant derivative for the third operator is the one with respect to $r$. The presence of $r$ is in the definition of $\gamma(\alpha)$ in \eqref{gamma}, so we rewrite $\mathcal{F}_3$ in a convenient way.

\begin{equation*}
\begin{split}
&\mathcal{F}_3(\theta_A,\tilde{\omega},r;B,p,\omega_0)=\frac{-\partial_{\alpha}\gamma_2(\alpha)}{2\pi}\int_{-\pi}^{\pi}\frac{-\gamma_2(\alpha)+z_2(\alpha')}{(\gamma_1(\alpha)-z_1(\alpha'))^2+(\gamma_2(\alpha)-z_2(\alpha'))^2}\tilde{\omega}(\alpha)\,d\alpha'\\[3mm]
&\hspace{1cm}+\frac{\partial_{\alpha}\gamma_1(\alpha)}{2\pi}\int_{-\pi}^{\pi}\frac{\gamma_1(\alpha)-z_1(\alpha')}{(\gamma_1(\alpha)-z_1(\alpha'))^2+(\gamma_2(\alpha)-z_2(\alpha'))^2}\tilde{\omega}(\alpha)\,d\alpha'\\[3mm]
&\hspace{1cm}-\frac{\omega_0}{2\pi}\partial_{\alpha}\gamma_2(\alpha)\textrm{P.V.}\int_{-\pi}^{\pi}\log\sqrt{(\gamma_1(\alpha)-\gamma_1(\alpha'))^2+(\gamma_2(\alpha)-\gamma_2(\alpha'))^2}\hspace{0.2cm}\partial_{\alpha}\gamma_1(\alpha')\,d\alpha'\\[3mm]
&\hspace{1cm}+\frac{\omega_0}{2\pi}\partial_{\alpha}\gamma_1(\alpha)\textrm{P.V.}\int_{-\pi}^{\pi}\log\sqrt{(\gamma_1(\alpha)-\gamma_1(\alpha'))^2+(\gamma_2(\alpha)-\gamma_2(\alpha'))^2}\hspace{0.2cm}\partial_{\alpha}\gamma_2(\alpha')\,d\alpha'.
\end{split}
\end{equation*}
\bigskip

\noindent In order to simplify the computation we will define $\gamma(\alpha)=r\hspace{0.1cm}(\tilde{\gamma}_1(\alpha),\tilde{\gamma}_2(\alpha))$ and $\partial_{\alpha}\gamma(\alpha)=r\hspace{0.1cm}(\partial_{\alpha}\tilde{\gamma}_1(\alpha),\partial_{\alpha}\tilde{\gamma}_2(\alpha))$.
\medskip

\begin{equation*}
\begin{split}
&D_r\mathcal{F}_3=\frac{d}{d\mu}\left[\mathcal{F}_3(\theta_A,\tilde{\omega},r+\mu r_1;B,p,\omega_0)\right]_{|\mu=0}=\\[3mm]
&=-\frac{r_1\partial_{\alpha}\tilde{\gamma}_2(\alpha)}{2\pi}\int_{-\pi}^{\pi}\frac{-r\tilde{\gamma}_2(\alpha)+z_2(\alpha')+1}{(r\tilde{\gamma}_1(\alpha)-z_1(\alpha'))^2+(r\tilde{\gamma}_2(\alpha)-z_2(\alpha')-1)^2}\tilde{\omega}(\alpha')\,d\alpha'\\[3mm]
&-\frac{r\partial_{\alpha}\tilde{\gamma}_2(\alpha)}{2\pi}\left\{\int_{-\pi}^{\pi}\frac{-r_1\tilde{\gamma}_2(\alpha)}{(r\tilde{\gamma}_1(\alpha)-z_1(\alpha'))^2+(r\tilde{\gamma}_2(\alpha)-z_2(\alpha')-1)^2}\tilde{\omega}(\alpha')\,d\alpha'\right.\\[3mm]
&\hspace{2cm}-\int_{-\pi}^{\pi}\frac{-r\tilde{\gamma}_2(\alpha)+z_2(\alpha')+1}{\left[(r\tilde{\gamma}_1(\alpha)-z_1(\alpha'))^2+(r\tilde{\gamma}_2(\alpha)-z_2(\alpha')-1)^2\right]^2}\tilde{\omega}(\alpha')\\[3mm]
&\hspace{3cm}\left.\cdot\left[2(r\tilde{\gamma}_1(\alpha)-z_1(\alpha'))r_1\tilde{\gamma}_1(\alpha)+2(r\tilde{\gamma}_2(\alpha)-z_2(\alpha')-1)r_1\tilde{\gamma}_2(\alpha)\right]\,d\alpha'\right\}\\[4mm]
&+\frac{r_1\partial_{\alpha}\tilde{\gamma}_1(\alpha)}{2\pi}\int_{-\pi}^{\pi}\frac{r\tilde{\gamma}_1(\alpha)-z_1(\alpha')}{(r\tilde{\gamma}_1(\alpha)-z_1(\alpha'))^2+(r\tilde{\gamma}_2(\alpha)-z_2(\alpha')-1)^2}\tilde{\omega}(\alpha')\,d\alpha'\\[3mm]
&+\frac{r\partial_{\alpha}\tilde{\gamma}_1(\alpha)}{2\pi}\left\{\int_{-\pi}^{\pi}\frac{r_1\tilde{\gamma}_1(\alpha)}{(r\tilde{\gamma}_1(\alpha)-z_1(\alpha'))^2+(r\tilde{\gamma}_2(\alpha)-z_2(\alpha')-1)^2}\tilde{\omega}(\alpha')\,d\alpha'\right.\\[3mm]
&\hspace{2cm}-\int_{-\pi}^{\pi}\frac{r\tilde{\gamma}_1(\alpha)-z_1(\alpha')}{\left[(r\tilde{\gamma}_1(\alpha)-z_1(\alpha'))^2+(r\tilde{\gamma}_2(\alpha)-z_2(\alpha')-1)^2\right]^2}\tilde{\omega}(\alpha')\\[3mm]
&\hspace{3cm}\left.\cdot\left[2(r\tilde{\gamma}_1(\alpha)-z_1(\alpha'))r_1\tilde{\gamma}_1(\alpha)+2(r\tilde{\gamma}_2(\alpha)-z_2(\alpha')-1)r_1\tilde{\gamma}_2(\alpha)\right]\,d\alpha'\right\}
\end{split}
\end{equation*}
\begin{equation*}
\begin{split}
&-\frac{\omega_0}{2\pi}2 r r_1\partial_{\alpha}\tilde{\gamma}_2(\alpha)\textrm{P.V.}\int_{-\pi}^{\pi}\log\left(r\sqrt{(\tilde{\gamma}_1(\alpha)-\tilde{\gamma}_1(\alpha'))^2+(\tilde{\gamma}_2(\alpha)-\tilde{\gamma}_2(\alpha'))^2}\right)\partial_{\alpha}\tilde{\gamma}_1(\alpha')\,d\alpha'\\[3mm]
&-\frac{\omega_0}{2\pi}r r_1\partial_{\alpha}\tilde{\gamma}_2(\alpha)\int_{-\pi}^{\pi}\partial_{\alpha}\tilde{\gamma}_1(\alpha')\,d\alpha'\\[3mm]
&+\frac{\omega_0}{2\pi}2rr_1\partial_{\alpha}\tilde{\gamma}_1(\alpha)\textrm{P.V.}\int_{-\pi}^{\pi}\log\left(r\sqrt{(\tilde{\gamma}_1(\alpha)-\tilde{\gamma}_1(\alpha'))^2+(\tilde{\gamma}_2(\alpha)-\tilde{\gamma}_2(\alpha'))^2}\right)\partial_{\alpha}\tilde{\gamma}_2(\alpha')\,d\alpha'\\[3mm]
&+\frac{\omega_0}{2\pi}r r_1\partial_{\alpha}\tilde{\gamma}_1(\alpha)\int_{-\pi}^{\pi}\partial_{\alpha}\tilde{\gamma}_2(\alpha')\,d\alpha'
\end{split}
\end{equation*}
\bigskip

\begin{remark}
We notice that all the terms above for $r=0$ disappear except for the first one and the third one. Moreover, by computing them at the Crapper point $(\theta_c, \tilde{\omega}_c)$ it follows that also the third will be zero because of the parity of the Crapper curve $z^c(\alpha)$ (see \eqref{z-parity}) and of $\tilde{\omega}(\alpha)$ which is even. Hence, in order to have the Fr\'echet derivative different from zero for every $\alpha\in[-\pi,\pi]$, we will choose $\gamma(\alpha)$ as \eqref{gamma} so that the first term will always be different from zero.
\end{remark}
\bigskip

\noindent Then we end up in

\begin{equation*}
\begin{split}
D_r\mathcal{F}_3(\theta_c,\tilde{\omega}_c,0;0,0,0,0)=&-\frac{r_1\partial_{\alpha}\tilde{\gamma}_2(\alpha)}{2\pi}\int_{-\pi}^{\pi}\frac{z^c_2(\alpha')+1}{(z^c_1(\alpha'))^2+(z^c_2(\alpha')+1)^2}\omega(\alpha')\,d\alpha'.
\end{split}
\end{equation*}
\medskip

\noindent It remains to prove the invertibility of the derivatives. In particular, the derivatives' matrix is the following

\begin{equation}\label{frechet-derivative}
D\mathcal{F}(\theta_c,\tilde{\omega}_c, 0;0,0,0)=
\begin{pmatrix}
D_{\theta_A}\mathcal{F}_1 & 0 & 0 \\
D_{\theta_A}\mathcal{F}_2 &   D_{\tilde{\omega}}\mathcal{F}_2 & 0\\
D_{\theta_A}\mathcal{F}_3 &  D_{\tilde{\omega}}\mathcal{F}_3 & D_r \mathcal{F}_3
\end{pmatrix}=
\begin{pmatrix}
\Gamma & 0 & 0\\
D_{\theta_A}\mathcal{F}_2 & \mathcal{A}(z^c(\alpha))+\mathcal{I}& 0\\
D_{\theta_A}\mathcal{F}_3 & D_{\tilde{\omega}}\mathcal{F}_3 & D_r\mathcal{F}_3
\end{pmatrix}\cdot 
\begin{pmatrix}
\theta_1\\
\omega_1\\
r_1
\end{pmatrix}
\end{equation} 
\medskip

\noindent where 

\begin{align*}
&\Gamma\theta_1=\cosh(\mathcal{H}\theta_c)\mathcal{H}\theta_1+q\frac{d}{d\alpha}\theta_1\\[3mm]
&(\mathcal{A}(z^c(\alpha))+\mathcal{I})\omega_1=2BR(z^c(\alpha),\omega_1)\cdot\partial_{\alpha}z^c(\alpha)+\omega_1.\\[3mm]
&D_r\mathcal{F}_3(\theta_c,\tilde{\omega}_c,0;0,0,0)=-\frac{r_1\partial_{\alpha}\tilde{\gamma}_2(\alpha)}{2\pi}\int_{-\pi}^{\pi}\frac{z^c_2(\alpha')+1}{(z^c_1(\alpha'))^2+(z^c_2(\alpha')+1)^2}\omega(\alpha')\,d\alpha'.
\end{align*}
\medskip

\noindent We put in evidence only these three operators since the matrix is diagonal and it will be invertible if the diagonal is invertible.\\

\noindent We observe immediately that the choice of the curve $\gamma(\alpha)$ is crucial since the second component of  $\partial_{\alpha}\tilde{\gamma}(\alpha)\neq 0$, for every $\alpha\in [-\pi,\pi]$. So we can invert $D_r\mathcal{F}_3(\theta_c,\tilde{\omega}_c,0;0,0,0)$, as required. For the other two operators we have to use Lemma \ref{DF1-invertible} and Lemma \ref{DF2-invertible} to overcome the problem of the non invertibility of $\Gamma$, see section \ref{point-existence-proof}. Hence, we state the following result.

\begin{theorem}
Let $|A|<A_0$. Then
\medskip

\begin{enumerate}
\item there exists $(\omega_0,\theta_A)$ and a unique smooth function $\tilde{\tau}^*:U_{\omega_0, \theta_A}\rightarrow H^2_{even}$, such that $\tilde{\tau}^*(0,\theta_A)=\tilde{\tau}$ (see \eqref{G-operator}),\\[3mm]

\item there exists $(B,p,\omega_0)$  and a unique smooth function $B^*:U_{p,\omega_0}\rightarrow U_B$, such that  $B^*(0,0)=0$,\\[3mm]

\item there exists a unique smooth function $\Theta_c: U_{B,p,\omega_0}\rightarrow H^2_{odd}\times H^1_{even}\times\mathbb{R}$, such that $\Theta_c(0,0,0)=(\theta_c,\tilde{\omega}_c,0)$
\end{enumerate}
\medskip

 and satisfy

$$\mathcal{F}(\Theta_c(B^*(p,\omega_0), p,\omega_0,);B^*(p,\omega_0),p,\omega_0)=0.$$
\end{theorem}
\medskip

\noindent The proof of Theorem \ref{patch-existence} holds directly from this Theorem.
\bigskip

\subsection*{Acknowledgements.} 
This work is supported in part by the Spanish Ministry of Science and Innovation, through the “Severo Ochoa Programme for Centres of Excellence in R\&D (CEX2019-000904-S)” and MTM2017-89976-P. DC and EDI
were partially supported by the ERC Advanced Grant 788250.

\printbibliography

@Article{AAW2013,
author = {Akers, B. F. and Ambrose, D. M. and  Wright, J. D.},
title = {Gravity perturbed Crapper waves},
journal = { Proc. R. Soc. Lond. Ser. A},
year = {2013},
volume = {470}
}

@Article{ASW2014,
author = {Ambrose, D. M. and Strauss, W. A. and Wright, J. D. },
title = {Global bifurcation theory for periodic traveling interfacial gravity-capillary waves},
journal = {Annales de l'Institut Henri Poincaré C, Analyse non linéaire},
volume = {33},
number = {4},
pages = {1081-1101},
year = {2016}
}

@Article{Babenko87,
author = {Babenko, K. I.},
title = {Some remarks on the theory of surface waves of finite amplitude},
journal = {Dokl. Akad. Nauk,},
year = {1987},
volume = {294},
pages = {1033–1037},
}

@Book{Brezis,
author = {Brezis, H},
title = {Functional Analysis, Sobolev Spaces and Partial Differential Equationd},
publisher = {Springer-Verlag New York},
year = {2011},
OPTkey = {•}
}

@Article{CCG2011,
author = {C\'ordoba, A. and C\'ordoba, D. and Gancedo F.},
title = {Interface evolution: the Hele-Shaw and Muskat problem},
journal = {Ann. of Math.},
year = {2011},
volume = {173},
pages = {477-542}
}

@Article{CEG2016,
author = {C\'ordoba, D. and Enciso, A. and Grubic, N.},
title = {On the existence of stationary splash singularities for the Euler equations},
journal = {Adv. Math.},
year = {2016},
volume = {288},
pages = {922-941},
}

@Article{CEG2019,
author = {C\'ordoba, D. and Enciso, A. and Grubic, N.},
title = {Self-intersecting interfaces for stationary solutions of the two-fluid Euler equations},
journal = {Ann. PDE},
year = {2021},
volume = {7},
number = {12}
}

@Article{Crapper1957,
author = {Crapper, G. D.},
title = {An exact solution for the progressive capillary waves of arbitrarly amplitude},
journal = {J. Fluid Mech.},
year = {1957},
volume = {2},
pages = {532-540}
}

@Article{CS2004,
author = {Constantin, A. and  Strauss, W. A. },
title = {Exact steady period water waves with vorticity},
journal = {Commun. Pure Appl. Math.},
year = {2004},
volume = {57},
pages = {481-527}
}

@Article{CSV2016,
author = {Constantin, A. and  Strauss, W. A. and Varvaruca, E.},
title = {Global bifurcation of a steady gravity water waves with
critical layer},
journal = {Acta Math.},
year = {2016},
volume = {217},
pages = {195-262}
}

@Article{CV2011,
author = {Constantin, A. and Varvaruca, E.},
title = {Steady periodic water waves with constant vorticity: regularity and local bifurcation},
journal = {Arch. Ration. Mech. Anal.},
year = {2011},
volume = {199},
pages = {33-67}
}

@Unpublished{deBoeck2014,
author = {\de{Boeck}, P.},
title = {Existence of capillary-gravity waves that are perturbations of Crapper waves},
note = {ArXiv: 1404.6189v1},
year = {2014}
}

@Unpublished{EWZ2019,
author = {Ehrnstr$\ddot{\textrm{o}}$m, M. and Walsh, S. and Zeng, C.},
title = {Smooth stationary water waves with exponentially localized vorticity},
note = {ArXiv:1907.07335v2},
year = {2019},
}

@Article{Hur-Vanden2020,
author = {Hur, V. M. and Vanden-Broeck, J. M.},
title = {A new application of Crapper's exact solution to waves in constant vorticity flows},
journal = {Eur. J. Mech. (B/Fluids)},
year = {2020}
}

@Article{Hur-Wheeler2020,
author = {Hur, V. M. and Wheeler, M. H.},
title = {Exact free surface in constant vorticity flows},
journal = {J. Fluid Mech.},
year = {2020},
volume={896}
}

@Article{Kinnersley1976,
author = {Kinnersley, W.},
title = {Exact large amplitude capillary waves on sheets of fluid},
journal = {J. Fluid Mech.},
year = {1976},
volume = {77},
pages = {229-241}
}

@book{OS2001,
author={Okamoto, H. and Shoji, M.},
title={The Mathematical Theory of Permanent Progressive Water Waves},
year={2001},
publisher={World Scientific Publishing, Singapore},
}

@Article{SWZ2013,
author = {Shatah, J. and Walsh, S. and Zheng, C.},
title = {Travelling water waves with compactly supported vorticity},
journal = {Nonlinearity},
year = {2013},
volume = {26},
pages = {1529-1564}
}

@Article{Wahlen2006-1,
author = {Wahl\'en,  E.},
title = {Steady periodic capillary waves with vorticity},
journal = {Ark. Mat.},
year = {2006},
volume = {44},
pages = {367-387}
}

@Article{Wahlen2006-2,
author = {Wahl\'en,  E.},
title = {Steady periodic capillary-gravity waves with vorticity},
journal = {SIAM J. Math. Anal.},
year = {2006},
volume = {38},
number={3},
pages = {921-943}
}

@Article{Zeidler,
author = {Zeidler, E.},
title = {Existenzbeweis f$\ddot{\textrm{u}}$r permanente kapillar-schwerewellen mit allgemeinen wirbelverteilungen},
journal = {Arch. Ration. Mech. Anal.},
year = {1973},
volume = {50},
pages = {34-72}
}
\end{document}